\documentclass[11pt, reqno]{amsart}
\usepackage{latexsym, amsmath, amssymb, amsthm, mathtools}
\usepackage{cases, verbatim}
\usepackage[active]{srcltx}
\usepackage{esint}
\usepackage[colorlinks]{hyperref}
\usepackage{hypernat}
\usepackage{xcolor}
\usepackage{float}
\usepackage[normalem]{ulem}
\usepackage{cancel}
\usepackage{subcaption}

\usepackage[T1]{fontenc}
\usepackage{mathscinet}

\hypersetup{
	unicode=false,          
	pdftoolbar=true,        
	pdfmenubar=true,        
	pdffitwindow=false,     
	colorlinks=false,       
	linkcolor=red,          
	linkbordercolor=red,
	citecolor=green,        
	citebordercolor=green,
	filecolor=magenta,      
	urlcolor=cyan,           
	urlbordercolor={1 1 1}  
}

\usepackage[margin=1.5in]{geometry}

\newtheorem{thm}{Theorem}

\theoremstyle{remark}
\newtheorem{rmk}[thm]{Remark}
\newtheorem{example}[thm]{Example}

\theoremstyle{definition}
\newtheorem{defi}[thm]{Definition}

\numberwithin{thm}{section} 
\numberwithin{equation}{section}

\makeatletter

\newcommand{\Rmnum}[1]{\expandafter\@slowromancap\romannumeral #1@}
\makeatother

\def\Om{\Omega}

\def\R{{\mathbb R}}

\def\F{{\mathcal F}}

\def\N{{\mathbb N}}

\def\S{{\mathcal S}}

\def\O{{\mathcal O}}

\newcommand{\Oba}{\overline{\Omega}}

\newcommand{\vep}{\varepsilon}
\newcommand{\ol}{\overline}

\newcommand{\dive}{\operatorname{div}}
\newcommand{\tr}{\operatorname{tr}}

\newcommand{\bpm}{\begin{pmatrix}}
	\newcommand{\epm}{\end{pmatrix}}

\newcommand{\la}{\langle}
\newcommand{\ra}{\rangle}

\newcommand{\beq}{\begin{equation}}
	\newcommand{\eeq}{\end{equation}}

		plus 6pt minus 12pt
\newcommand{\kom}[1]{}
\renewcommand{\kom}[1]{{\bf [#1]}}
\renewcommand{\div}{\operatorname{div}}
\newcommand{\abs}[1]{\left|#1\right|}
\newcommand{\norm}[1]{\left\|#1\right\|}
\newcommand{\eps}{\varepsilon}
\newcommand{\Rn}{\mathbb{R}^n}

\def\red#1{\textcolor{red}{#1}}

\newcommand{\numberthis}{\addtocounter{equation}{1}\tag{\theequation}}
\newcommand{\ud}{\, d}

\def\vint_#1{\mathchoice%
	{\mathop{\kern 0.2em\vrule width 0.6em height 0.69678ex depth -0.58065ex
			\kern -0.8em \intop}\nolimits_{\kern -0.4em#1}}%
	{\mathop{\kern 0.1em\vrule width 0.5em height 0.69678ex depth -0.60387ex
			\kern -0.6em \intop}\nolimits_{#1}}%
	{\mathop{\kern 0.1em\vrule width 0.5em height 0.69678ex
			depth -0.60387ex
			\kern -0.6em \intop}\nolimits_{#1}}%
	{\mathop{\kern 0.1em\vrule width 0.5em height 0.69678ex depth -0.60387ex
			\kern -0.6em \intop}\nolimits_{#1}}}

\title[Quantitative stability]{Quantitative stability for quasilinear parabolic equations}

\author[T. Kurkinen]{Tapio Kurkinen}
\address[Tapio Kurkinen]{Geometric Partial Differential Equations Unit, Okinawa Institute of Science and Technology Graduate University, Okinawa 904-0495, Japan, {\tt tapio.kurkinen@oist.jp}}

\author[Q. Liu]{Qing Liu}
\address[Qing Liu]{Geometric Partial Differential Equations Unit, Okinawa Institute of Science and Technology Graduate University, Okinawa 904-0495, Japan, {\tt qing.liu@oist.jp}}

\date{\today}

\begin{document}

	\begin{abstract}
		We examine the stability of a class of quasilinear parabolic partial differential equations under perturbations. We are interested in the behavior of viscosity solutions as the perturbation parameter vanishes and establish explicit convergence rates by adapting standard comparison arguments. Despite the possible singular or degenerate nature of the parabolic operator, our framework covers, in particular, both the normalized and the variational $p$-parabolic equations, providing quantitative estimates for perturbations of the exponent $p$ and limits arising from regularized approximations. 
	\end{abstract}
	

	\maketitle

	\section{Introduction}
	\subsection{Motivation}
	In this paper we study the stability for viscosity solutions of a class of quasilinear parabolic partial differential equations, which take the form 
	\begin{equation}\label{general parabolic lim eq}
		\partial_t u-\tr(A(\nabla u)\nabla^2 u)+H(x, t, \nabla u)=0 \quad \text{in $\Omega\times (0, T)$,}
	\end{equation}
	where $\Omega$ is a bounded domain in $\R^n$, $T>0$ is fixed, and 
	\[
	A: \R^n\setminus \{0\}\to \S^n_+, \quad H: \Omega\times (0, T)\times \R^n\to \R,
	\]
	are given functions satisfying appropriate assumptions to be introduced later. Here, $\S^n_+$ represents the set of all nonnegative symmetric $n\times n$ matrices. A key feature of this general setting is that the elliptic operator may exhibit singularities at vanishing gradients. 
	
	One typical example of applicable equations we have in mind is the following parabolic normalized $p$-Laplace type equation (with $1<p<\infty$)
	\begin{equation}\label{np-parabolic0}
		\partial_t u-\Delta^N_p u=0 \quad \text{in $\Om_T:=\Omega\times (0, T)$,}
	\end{equation}  
	where
	\[
	\Delta^N_p u:= |\nabla u|^{2-p}\dive(|\nabla u|^{p-2}\nabla u).
	\]
	Another example is the parabolic equation with the variational $p$-Laplacian:
	\begin{equation}\label{vp-parabolic0}
		\partial_t u-\Delta_p u=0 \quad \text{in $\Om_T$,}
	\end{equation}
	where 
	\[
	\Delta_p u:= \dive(|\nabla u|^{p-2}\nabla u). 
	\]
	We consider the associated Cauchy-Dirichlet boundary value problem for these parabolic equations. The parabolic boundary condition is given by 
	\begin{equation}\label{parabolic bdry-cond}
		u=g \quad \text{on $\partial_p \Om_T$}
	\end{equation}
	for
	\[
	\partial_p \Om_T=(\Omega\times \{0\})\cup (\partial \Omega\times [0, T)),
	\]
	where $g: \partial \Omega\to \R$ is assumed to be continuous.
	
	One of the main questions addressed in this paper concerns the stability of $p$-Laplace type equations under perturbations of the parameter $p$. For a fixed parabolic boundary value $g$, it is known that, as $p\to q$ for some $q\in(1,\infty)$, the unique solution $u_p$ of the Cauchy-Dirichlet problem \eqref{np-parabolic0} or \eqref{vp-parabolic0} converges uniformly to the solution $u_q$ of the corresponding problem with parameter $q$. In the framework of viscosity solutions, consult \cite[Theorem~6.1]{OS} and \cite[Theorem~2.2.1]{Gbook} for such qualitative stability results on the variational and normalized $p$-Laplace equations of parabolic type. See also \cite{KP, LP} for related analysis with different approaches. However, these works do not provide an explicit rate of convergence. Our goal is therefore to further quantify the uniform convergence 
	and to discuss the optimality of the resulting estimates.
	
	On the other hand, when studying the solutions $u$ of singular parabolic equations such as \eqref{np-parabolic0} and \eqref{vp-parabolic0}, it is common to approximate the solutions of regularized equations of the form 
	\begin{equation}\label{regularized eq}
		\partial_t u-\left(|\nabla u|^2+\vep^2\right)^{\frac{p'-p}{2}}\div\left(\left(|\nabla u|^2+\vep^2\right)^{\frac{p-2}{2}}\nabla u\right)=0 \quad \text{in $\Omega_T$,}
	\end{equation}
	where $\vep>0$ is small. Here, the choices $p'=2$ and $p'=p$ yield regularized approximations of \eqref{np-parabolic0} and \eqref{vp-parabolic0}, respectively. As $\vep\to 0$,  the comparison principle for the limiting equation  yields the uniform convergence of the solution $u_\vep$ of \eqref{regularized eq} to the solution $u$ of the corresponding singular equation. It is then natural to ask about the rate of this convergence in applications. The case with $p'=2$ and $p=1$ corresponds to the level-set mean curvature flow equation, for which such a regularization was introduced in \cite{ES}. The convergence rate in this case was essentially established by Deckelnick \cite{De}, and later proved by Mitake \cite{Mi} using a different method based on comparison arguments. Addressing convergence rates for approximations as in \eqref{regularized eq} is another main objective of the present paper.
	
	
	The method employed in this paper builds on the arguments developed by Mitake in \cite{Mi}, as well as earlier works by Cockburn, Gripenberg and Londen \cite{CGL} and Jakobsen and Karlsen \cite{JK1}, which established general continuous dependence estimates for viscosity solutions of nonlinear parabolic equations. Further related results include extensions to nonlocal PDEs in \cite{JK2} and to general Neumann-type boundary value problems in \cite{JG}. While it may be possible to adapt these methods to study quantitative stability for $p$-Laplace type equations in regimes without vanishing-gradient singularities, precise estimates do not appear to have been explicitly given, to the best of our knowledge. Also, the singular case, such as \eqref{vp-parabolic0} with $1<p<2$, seems unexplored and constitutes one of the focuses of the present work. 
	
	A recent paper of Bungert \cite{Bun} studies the convergence rate of $p$-harmonic functions to infinity-harmonic functions, using partly the variational structure of the $p$-Laplace operator. In contrast, based on the viscosity solution theory, we here provide a purely PDE approach to show the convergence rate for a general class of related parabolic stability problems. Our method further develops the techniques in \cite{CGL, JK1, Mi} and extends them to treat convergence rates arising from more general stability and regularization problems.

	As explained through the examples above, the equations of interest may exhibit singularity at vanishing gradients. To handle such singularities, we employ a generalized notion of viscosity solutions, known as $\mathcal{F}$-solutions. A brief overview of this approach is given in Section~\ref{sec:preli}; for a detailed introduction, we again refer the reader to \cite{OS, Gbook}.

	\subsection{Main results}
	Let us now introduce more details about our setting. Let $T>0$ and $\Om\subset\Rn$ be a bounded domain. In the sequel, we set $\Om_T=\Omega\times (0, T)$ and $\Rn_0=\R^n\setminus \{0\}$. Consider a family of parabolic equations
	\begin{equation}\label{general parabolic approx eq}
		\partial_t u-\tr(A_\vep(\nabla u)\nabla^2 u)+H_\vep(x, t, \nabla u)=0 \quad \text{in $\Om_T$}
	\end{equation}
	for $\vep\geq 0$, where $A_\vep: \Rn_0\to \S^n_+$ and $H_\vep:\Omega_T \times \R^{n}\mapsto \R$  are continuous functions satisfying appropriate assumptions to be introduced below. 
	
	
	Let $\partial_p \Om_T$ denote the parabolic boundary of $\Omega\times (0, T)$,  defined by
	\[
	\partial_p \Om_T=(\Omega\times \{0\})\cup (\partial \Omega\times [0, T)).
	\]
	We impose the initial and boundary condition 
	\begin{equation}\label{parabolic approx bdry-cond}
		u =g_\vep \quad \text{on $\partial_p \Om_T$}
	\end{equation}
	for \eqref{general parabolic approx eq}, where $g_\vep\in C(\partial \Om_T)$ are given for $\vep\geq 0$.   
	
	
	Suppose that $u_\vep$ (with $\vep\geq 0$) is a unique viscosity solution of \eqref{general parabolic approx eq}\eqref{parabolic approx bdry-cond}. Our main result quantifies the convergence of $u_\vep\to u_0$ in terms of $\vep>0$ under certain assumptions on the convergence rate of $A_\vep\to A_0$, $H_\vep\to H_0$, and $g_\vep \to g_0$ as $\vep\to 0$.
	
	Now let us state our assumptions on $A_\vep$ and $H_\vep$ with $\vep\geq 0$. 
	Note first that for a matrix $M\in \S^n_+$, one can find its square root, denoted by $M^{1/2}$, which is also nonnegative. There thus exists the matrix valued function ${A_\vep}^{1/2}: \Rn_0\to \S^n_+$ for $\vep\geq 0$. We impose the following assumption on $A_\vep$ and $H_\vep$.
	\begin{enumerate}
		\item[(A1)] There exists $k>2$ such that for any $\vep\geq 0$ small, $f(x)=|x|^k$ satisfies 
		\begin{equation}\label{singular-cancel}
			\tr \left(A_\vep(\nabla f(x)) \nabla^2 f(x)\right)\to 0 \quad \text{as $x\to 0$.}
		\end{equation} 
	\end{enumerate}
	\begin{enumerate}
		\item[(A2)] $H_\vep: \Om_T\times \R^n\to \R$ is continuous, and there exists $L_H\geq0$ such that 
		\[
		|H_\vep(x_1, t_1, \xi)-H_\vep(x_2, t_2, \xi)|\leq L_H(1+|\xi|)\left(|x_1-x_2|+\abs{t_1-t_2}\right)
		\]
		for all $(x_1, t_1),\ (x_2, t_2)\in \Omega_T$,  $\xi\in \R^n$ and all $\vep\geq 0$. 
		
	\end{enumerate}
	These assumptions guarantee the existence and uniqueness of continuous $\F$-solutions to the associated Cauchy-Dirichlet problem; see \cite{Gbook}. The existence of $k>2$ required in (A1) is an assumption related to possible singularities of the equation at vanishing gradients. A typical example of singular operators we will consider in our later study is the case of  the stability of the $p$-Laplacian for $p$ near a fixed $q\in (1, 2)$, where $A_\vep$ is taken to be 
	\begin{equation}\label{vp-operator}
		A_\vep(\xi)= |\xi|^{q+\vep-2} \left(I+(q+\vep-2) \frac{\xi\otimes \xi}{|\xi|^2}\right)
	\end{equation}
	for $\vep\geq 0$. In this case, one can choose  $k>q/(q-1)$ for $\vep\geq 0$ sufficiently small.

	For general $A_\vep$ and $H_\vep$, we next impose assumptions on the convergence of $A_\vep\to A_0$, $H_\vep\to H_0$ as $\vep\to 0$: 
	\begin{enumerate}
		\item[(C1)] There exist $c_A>0$ and $\alpha>0, \beta>\frac{2-k}{2(k-1)}$ such that 
		\begin{equation}\label{eq:c1}
			\|A_\vep(\xi)^{1/2} -A_0(\xi)^{1/2}\|\leq c_A\vep^{\alpha} (1+|\xi|^\beta)
		\end{equation}
		holds for all $\vep>0$ small and $\xi\in \Rn_0$. Here, $k$ is the constant given in (A1). 
		\item[(C2)] There exist $c_H>0$ and $\gamma>0$  such that 
		\[
		|H_\vep(x, t, \xi)-H_0(x, t, \xi)|\leq c_H \vep^\gamma 
		\] 
		holds for all $\vep>0$ small and $(x, t)\in \Om_T,  \xi\in \R^n$.
	\end{enumerate}
	Note that by the condition $\beta>\frac{2-k}{2(k-1)}$ in (C1), we have $\beta>-1/2$ for $k>2$. It is worth remarking that one can generalize (C1) by assuming the existence of $\Sigma_\vep: \R^n_0\to \R^{m\times n}$ with $m\in \N$ for all $\vep\geq 0$ small such that $\Sigma_\vep^T\Sigma_\vep=A_\vep$ in $\R^n_0$ and \eqref{eq:c1} holds for $\Sigma_\vep$ in place of $A_\vep^{1/2}$. For clarity and simplicity, we restrict our attention to the case $\Sigma_\vep=A_\vep^{1/2}$, which already covers the applications of interest.  
	
	Let us present our main theorem.
	

	\begin{thm}[Quantitative stability for parabolic equations]\label{thm:parabolic}
		Let $T>0$ and $\Omega$ be a bounded domain in $\R^n$. For $\vep\geq 0$, assume that $A_\vep: \R^n_0\to \S_+^n$ and ${H_\vep: \Omega_T\times \R^n\to \R}$ are continuous functions satisfying (A1)(A2) and (C1)(C2). Let $g_\vep\in C(\partial_p \Om_T)$ and $u_\vep\in C(\Oba\times [0, T))$ be bounded $\F$-solutions of  \eqref{general parabolic approx eq}\eqref{parabolic approx bdry-cond}. Assume in addition that $u_\vep$ are equi-H\"older continuous in $\ol{\Om_T}$, that is, there exist $L>0$ and $\theta\in (0, 1]$ such that,   
		\begin{equation}\label{equi-holder}
			|u_\vep(x_1, t_1)-u_\vep(x_2, t_2)|\leq L(|x_1-x_2|^\theta +|t_1-t_2|^\theta) 
		\end{equation}
		for all $(x_1, t_1), (x_2, t_2)\in \Oba\times [0, T)$ and $\vep\geq 0$ small. 
		Then, for all $\vep>0$ small, 
		\begin{equation}\label{error-est parabolic}
			\sup_{\Oba\times [0, T)}|u_\vep-u_0|\leq \sup_{\partial_p \Om_T}|g_\vep-g_0|+C\vep^{\nu}  +c_H T\vep^\gamma,
		\end{equation}
		holds with 
		\begin{equation}\label{nu-exponent}
			\nu=\frac{\alpha\theta}{1+(1-\theta)\max\{\beta, 0\}},    
		\end{equation}
		where $C>0$ depends on $T, n, k, L_A, L_H, \beta, c_A$ as well as the H\"older exponent $\theta$ and the constant $L$. Here, $k>2$, $L_A\geq0$, $L_H\geq0$ appear in (A1)(A2), and $\alpha, \beta, \gamma, c_A, c_H\in \R$ are given in (C1)(C2). 
	\end{thm}
	
	This result applies to important quasilinear parabolic equations with vanishing gradient singularities, for which typically (C1) holds with $\beta<0$. While the parameter $\beta<0$ represents the singularity strength, the term $\max\{\beta, 0\}$ in \eqref{nu-exponent} indicates that our estimate \eqref{error-est parabolic} does not capture any influence of this singularity. Moreover, for equi-Lipschitz solutions, corresponding to the case $\theta=1$, we always have $\nu=\alpha$, and the behavior of $A_\vep(\xi)$ with respect to the gradient variable $\xi$ essentially plays no decisive role in determining the stability estimate.
	
	Our proof of Theorem~\ref{thm:parabolic} develops the standard comparison argument for possibly degenerate or singular parabolic equations, incorporating the H\"older regularity of solutions into the estimates. Without effecting these estimates, this result can be readily adapted to the spatially periodic setting by simply taking $\Omega$ to be a torus; see Remark~\ref{rmk:periodic} for more details. Moreover, the assumption of H\"older continuity in time in \eqref{equi-holder} can be removed when vanishing gradients induce only a mild singularity of the equation. Such an improvement follows from the use of the parabolic Crandall–Ishii lemma, which avoids doubling the time variables and thereby simplifies the proof.
	
	While the Hölder regularity of solutions is not the primary focus of this paper, our method relies on the equi-H\"older continuity of $\F$-solutions up to the boundary. In general, the validity of such estimates may depend on the equation structure, boundary data, and the geometry of $\Om$. In the case of $p$-Laplace type parabolic equations, we refer the reader to \cite{Do11, JS17, AP18, IJS19} for related interior regularity results, and to \cite{DZ24,Lee2025} for boundary estimates. Note that both bodies of literature are extensive, and we cite only a selection here. We also remark that the regularity theory for the general equation \eqref{general parabolic lim eq} remains only partially understood.

	We are also interested in the elliptic variant of such stability problems. Consider now
	\begin{equation}\label{general approx eq}
		-\tr(A_\vep(\nabla u)\nabla^2 u)+H_\eps(x, u, \nabla u)=0 \quad \text{in $\Omega$}
	\end{equation}
	with the Dirichlet boundary condition 
	\begin{equation}\label{approx bdry-cond}
		u =g_\vep \quad \text{on $\partial \Omega$}
	\end{equation}
	for given $g_\vep\in C(\partial \Omega)$. 
	While our approach to parabolic problems can be readily extended to the elliptic case under similar assumptions on $A_\vep$ and $H_\vep$, we need to additionally assume a monotonicity condition on $H_\vep$ with respect to $u$ for all $\vep\geq 0$. More precisely, we assume that there exists $\lambda>0$ such that 
	\begin{equation}\label{mono-0}
		\rho\mapsto H_\vep(x, \rho, \xi)-\lambda \rho \quad \text{is nondecreasing in $\R$. }
	\end{equation}
	We obtain the same convergence rate as in Theorem \ref{thm:parabolic}; see Theorem \ref{thm:elliptic}. However, it is not clear to us how to drop or relax the monotonicity assumption \eqref{mono-0}.

	\subsection{Applications}
	
	Our general framework in Theorem~\ref{thm:parabolic} yields stability estimates for a broad class of quasilinear parabolic equations including the normalized \eqref{np-parabolic0} $p$-Laplace equation and variational $p$-Laplace equation \eqref{vp-parabolic0} and their further generalizations, as well as regularizing approximations for degenerate $p$-Laplace equations ($1\leq p\leq \infty$). Below we summarize our major convergence rate results obtained for these equations. See Section~\ref{sec:app} for precise statements, detailed analysis, concrete examples, and further discussions. 
	
	\begin{itemize}
		\item (Normalized $p$-Laplacian) Let $u_p$ be a solution of \eqref{np-parabolic0}\eqref{parabolic bdry-cond} for $1\leq p<\infty$. For any fixed $1\leq q<\infty$, if $u_p$ are spatially equi-H\"older continuous with exponent $\theta\in (0, 1]$ for all $p$ near $q$, then $u_p\to u_q$ uniformly as $p\to q$ with convergence rate $O(|p-q|^\theta)$. 
		\item (Variational $p$-Laplacian) Let $u_p$ be a solution of \eqref{vp-parabolic0}\eqref{parabolic bdry-cond} for $1<p<\infty$. For any fixed $1\leq q<\infty$, if $u_p$ are equi-H\"older continuous with exponent $\theta\in (0, 1]$ for all $p$ near $q$, then, depending on $p$, $u_p\to u_q$ uniformly as $p\to q$ with convergence rate 
		\begin{enumerate}
			\item[(i)] $O(|p-q|^\theta)$ when $p<2$ (and hence $q\leq 2$);
			\item[(ii)] $O(|p-q|^\nu)$ with any 
			\[
			0<\nu< \frac{2\theta}{(1-\theta)q+2\theta}
			\] 
			when $p>2$ (and hence $q\geq 2$).
		\end{enumerate}
		\item (Regularization for generalized $p$-Laplacian) Let $p\geq 1$, $p'\geq 2$. For all $\vep>0$ small, let $u_\vep$ be a solution of \eqref{regularized eq}
		and satisfy $u_\vep=g$ on $\partial_p \Omega_T$ for $g\in C(\partial_p\Omega_T)$. As $\vep\to 0$, $u_\vep\to u$ uniformly, where $u$ is the unique solution of
		\begin{equation*}
			\partial_t u-\abs{\nabla u}^{p'-p}\div\left(\abs{\nabla u}^{p-2}\nabla u\right)=0
		\end{equation*}
		satisfying the same boundary condition. If $u_\vep$ is spatially equi-H\"older continuous with exponent $\theta\in (0, 1]$ for $\vep>0$ small, then we obtain the rate $O(\vep^\nu)$ for the uniform convergence $u_\vep\to u$, where $\nu>0$ can be chosen as follows: 
		\begin{enumerate}
			\item[(i)] If $p'=2$, then $0<\nu< \theta/2$;
			\item[(ii)] If $2<p'< 3$, then $\nu=(p'-2)\theta$; 
			\item[(iii)] If $3\leq p'\leq 4$, then $0<\nu<\theta$;
			\item[(iv)] If $p'>4$, then 
			\[
			0<\nu< \frac{\theta}{1+(1-\theta)(p'-4)}. 
			\]
		\end{enumerate}
	\end{itemize}

	In addition to the applications above, our results can also be applied to study parabolic quantitative stability for general $p$-Laplace type equations that cover both the normalized and variational cases, as well as a regularizing approximation for biased infinity-Laplace equations arising in the stochastic tug-of-war games. Several examples for quasilinear elliptic equations are also discussed in Section~\ref{sec:app}.  
	
	Our main theorem applies to the vanishing viscosity limit for Hamilton-Jacobi equations as well. For $\vep\geq 0$ small, consider the solution $u_\vep$ to the initial boundary value problem for the Hamilton-Jacobi equation
	\[
	\partial_t u -\vep \Delta u+H(x, t, \nabla u)=0 \quad \text{in $\Omega_T$,}
	\]
	where the Hamiltonian $H$ is assumed to satisfy the same regularity assumption as in (A2). Then, for fixed initial and boundary data independent of $\vep$, Theorem~\ref{thm:parabolic} immediately yields the convergence rate $\|u_\vep-u_0\|_{L^\infty(\Omega_T)}=O(\vep^{1/2})$ for equi-Lipschitz solutions $u_\vep$ and the limit solution $u_0$ of the corresponding inviscid Hamilton-Jacobi equation. This recovers the classical result of \cite{Fl}, originally established using an approach based on differential games, and later obtained in \cite{CL1, CL2} via viscosity solution theory and in \cite{Tr} using the nonlinear adjoint method. As this application is not the main focus of the present paper, we do not pursue it further, but instead refer the reader to the aforementioned references, as well as \cite{QSTY, CD, CG} for more recent progress on improved rates in the convex Hamiltonian case. 
	
	In this work, we do not seek general optimality of the convergence rates. We only discuss the sharpness of our estimates through some examples, including in particular the stability of spatially periodic solutions to normalized $p$-Laplace equations under perturbations of $p$. The above convergence rates for the regularization of generalized $p$-Laplacian are consistent with the estimate obtained in \cite{Mi} for Lipschitz solutions of regularized level set curvature flow equations, which corresponds to the case $p=1$, $p'=2$ in our general setting. 
	
	\subsection*{Acknowledgments}
	
	The authors would like to thank Hiroyoshi Mitake, F\'elix del~Teso and Denis Brazke for helpful discussions. The work of QL was supported by JSPS Grant-in-Aid for Scientific Research, No.~22K03396.

	\section{Preliminaries}\label{sec:preli}
	In this section, we briefly review some basics about viscosity solutions of \eqref{general parabolic lim eq}. The notion of solutions we adopt for fully nonlinear singular equations is called $\F$-solutions, originated from the works \cite{CGG0, CGG, IS} mainly for geometric evolution equations. This framework was developed by \cite{OS} to treat $p$-Laplace type equations. We also refer to the monograph \cite{Gbook} for a comprehensive introduction. 
	
	Let $T>0$ and $\Omega\subset \R^n$ be an open set. Consider quasilinear parabolic equations of the form \eqref{general parabolic lim eq},
	where $A: \R^n_0\to \S^n_+$ and $H: \Om_T\times \R^n\to \R$ are assumed to be continuous. 
	
	Assume that there exists a nonempty set 
	\begin{equation}\label{f-fun}
		\F\subset\{f\in C^2([0, \infty)): f(0)=f'(0)=f''(0)=0, \ f''(r)>0, \ \forall \ r>0\}    
	\end{equation}
	such that for all $f\in \F$, $R_f(x)=\pm f(|x|)$ satisfy 
	\begin{equation}\label{singular-cancel0}
		\|A(\nabla R_f(x))\nabla^2 R_f(x))\|\to 0 \quad \text{as $x\to 0$.}
	\end{equation}
	In many of our applications, we can take $R_f$ of the form $|x|^k$ and ask the existence of $k>0$ that guarantees the condition \eqref{singular-cancel0}. 
	
	Let us now introduce the class of admissible test functions for \eqref{general parabolic lim eq}. 
	\begin{defi}[Compatible test functions]
		A function $\psi\in C^2(\Omega_T)$ is said to be \emph{compatible} with $A$ if for any $z_0=(x_0, t_0)\in \Omega_T$ with $\nabla \psi(z_0)=0$, there exists $f\in \F$ such that for any $z=(x, t)$ near $z_0$, 
		\begin{equation}\label{c2f}
			|\psi(z)-\psi(z_0)-\partial_t \psi(z_0)(t-t_0)|\leq f(|x-x_0|)+o(|t-t_0|).
		\end{equation}
		We denote by $C^2_A(\Omega_T)$ the set of all $\psi\in C^2(\Omega_T)$ that are compatible with $A$. These are also sometimes called admissible test functions in the literature.
	\end{defi}

	We next use the compatible test function class to define $\F$-solutions of \eqref{general parabolic lim eq}. 
	\begin{defi}\label{defi:f-sol}
		A locally bounded, upper semicontinuous function $u$ in $\Omega_T$ is called an \emph{$\F$-subsolution} of \eqref{general parabolic lim eq} if whenever there exist $z_0\in \Omega_T$ and $\psi\in C^2_A(\Omega_T)$ such that $u-\psi$ attains a local maximum at $z_0$, we have
		\begin{numcases}{}
			\partial_t \psi(z_0)-\tr (A(\nabla \psi(z_0))\nabla^2 \psi(z_0))+H(z_0, \nabla \psi(z_0))\leq 0 & \text{if $\nabla\psi(z_0)\neq 0$,}\label{sub off singular}\\
			\partial_t \psi(z_0)+H(z_0, \nabla \psi(z_0))\leq 0 & \text{if $\nabla\psi(z_0)=0$.} \label{sub on singular}
		\end{numcases}
		Analogously, a locally bounded, lower semicontinuous function $u$ in $\Omega_T$  is called an \emph{$\F$-supersolution} of \eqref{general parabolic lim eq} if whenever there exist $z_0\in \Omega_T$ and $\psi\in C^2_A(\Omega_T)$ such that $u-\psi$ attains a local minimum at $z_0$, we have
		\begin{numcases}{}
			\partial_t \psi(z_0)-\tr (A(\nabla \psi(z_0))\nabla^2 \psi(z_0))+H(z_0, \nabla \psi(z_0))\geq 0 & \text{if $\nabla\psi(z_0)\neq 0$,} \label{super off singular} \\
			\partial_t \psi(z_0)+H(z_0, \nabla \psi(z_0))\geq 0 & \text{if $\nabla\psi(z_0)=0$.} \label{super on singular}
		\end{numcases}
		A function $u\in C(\Omega_T)$ is called an \emph{$\F$-solution} of \eqref{general parabolic lim eq} if it is both an {$\F$-subsolution} and an {$\F$-supersolution} of \eqref{general parabolic lim eq}. 
	\end{defi}
	
	Note that if $A$ is continuous in $\R^n$, then the class $\F$ contains all functions $f(r)=r^k$ with $k>2$, and the corresponding $\F$-solutions are known to coincide with the standard viscosity solutions in this non-singular case; see for example \cite[Proposition~2.28]{Gbook} for the proof of this fact. Hence, Definition~\ref{defi:f-sol} can be viewed as a generalized notion of viscosity solutions that applies also to parabolic equations with singularities at vanishing gradients.  
	
	A typical example of such singular equations, as already mentioned in the introduction, is the parabolic $p$-Laplace equation \eqref{vp-parabolic0} for $p\in (1, 2)$. In this case, 
	\[
	A(\xi)= |\xi|^{p-2} \left(I+(p-2) \frac{\xi\otimes \xi}{|\xi|^2}\right), \quad \xi\in \R^n_0,
	\]
	and choosing $R_f(x)=\pm |x|^k$ satisfies \eqref{singular-cancel} provided that  $k>p/(p-1)$. In other words, for such $k$, the function $f(r)=r^k$ belongs to the class $\mathcal{F}$ and plays a key part in the construction of test functions to handle the singularities of the equation.

	
	Let us also include the elliptic version of the definition of $\F$-solutions for later use. Let $\Omega$ be a domain in $\R^n$ and consider the elliptic equation
	\begin{equation}\label{general elliptic}
		-\tr(A(\nabla u)\nabla^2 u)+H_\eps(x, u, \nabla u)=0 \quad \text{in $\Omega$,}
	\end{equation}
	where $A: \R^n_0\to \S^n_+$ and $H: \Omega\times \R\times \R^n\to \R$ are continuous. In this case, we continue to assume that \eqref{singular-cancel0} holds with a nonempty class $\F$. The associated class of admissible test functions is defined by simply suppressing the time variable in the parabolic case. More precisely, let $C_A^2(\Omega)$ denote the set of all test functions compatible with $A$ in this case. We say $\psi\in C^2_A(\Omega)$ if for any $x_0\in \Omega$ satisfying $\nabla\psi(x_0)=0$, there exists $f\in \F$ such that 
	\[
	|\psi(x)-\psi(x_0)|\leq f(|x-x_0|).
	\]
	
	\begin{defi}\label{defi:f-sol elliptic}
		A locally bounded, upper semicontinuous function $u$ in $\Omega$ is called an \emph{$\F$-subsolution} of \eqref{general elliptic} if whenever there exist $x_0\in \Omega$ and $\psi\in C^2_A(\Omega)$ such that $u-\psi$ attains a local maximum at $x_0$, we have
		\begin{numcases}{}
			-\tr (A(\nabla \psi(x_0))\nabla^2 \psi(x_0))+H(x_0, u(x_0), \nabla \psi(x_0))\leq 0 & \text{if $\nabla\psi(x_0)\neq 0$,}\notag \\
			H(x_0, u(x_0), \nabla \psi(x_0))\leq 0 & \text{if $\nabla\psi(x_0)=0$.} \notag
		\end{numcases}
		Analogously, a locally bounded, lower semicontinuous function $u$ in $\Omega$  is called an {$\F$-supersolution} of \eqref{general elliptic} if whenever there exist $x_0\in \Omega$ and $\psi\in C^2_A(\Omega)$ such that $u-\psi$ attains a local minimum at $x_0$, we have
		\begin{numcases}{}
			-\tr (A(\nabla \psi(x_0))\nabla^2 \psi(x_0))+H(x_0, u(x_0), \nabla \psi(x_0))\geq 0 & \text{if $\nabla\psi(x_0)\neq 0$,}\notag \\
			H(x_0, u(x_0), \nabla \psi(x_0))\geq 0 & \text{if $\nabla\psi(x_0)=0$.} \notag
		\end{numcases}
		A function $u\in C(\Omega)$ is called an {$\F$-solution} of \eqref{general elliptic}  if it is both an {$\F$-subsolution} and an {$\F$-supersolution} of \eqref{general elliptic}. 
	\end{defi}

	\section{Quantitative stability}

	Let us prove Theorem \ref{thm:parabolic} by developing the doubling variable technique that is often used in the proof of comparison principles.  
	
	\begin{proof}[Proof of Theorem \ref{thm:parabolic}]
		Below let us show in detail that, for any $\vep>0$ and all $t\in[0,T)$, 
		\begin{equation}\label{error-est1}
			\sup_{\ol{\Om_T}} (u_\vep-u_0)\leq \sup_{\partial_p\Om_T} |g_\vep-g|+C\vep^{\nu}+c_H \vep^\gamma
		\end{equation}
		holds for some $C>0$. We can exchange $u_0$ and $u_\vep$ to obtain the same estimate for $\sup_{\ol{\Om_T}} (u_0-u_\vep)$.
		
		
		\noindent {\bf Step 1. Doubling variables.}  For $\vep>0$, we assume that $\sup_{\ol{\Om_T}} (u_\vep-u_0)>0$, for otherwise we immediately obtain the desired estimate. Let $b>0$, $M>0$ to be chosen later and denote $K_T=\ol{\Omega}\times [0, T)$. Define
		\begin{equation*}
			\Psi(x,t)=u_\eps(x,t)-u_0(x,t)-Mt-\frac{2b}{T-t}, \quad (x,t)\in K_T.
		\end{equation*}
		By the boundedness of $u_\eps$, there exists a maximizer $(\hat{x},\hat{t})\in K_T$ of $\Psi$, that is, 
		\begin{equation*}
			\Psi(\hat{x},\hat{t})=\max_{(x, t)\in \Oba\times [0, T)}\Psi(x,t).
		\end{equation*}
		Based on the standard doubling variable technique, we define
		\begin{equation}\label{double variable1}
			\Psi_\delta(x,t,y,s)=u_\eps(x,t)-u_0(y,s)-\frac{\abs{x-y}^k}{\delta^k}-\frac{\abs{t-s}^k}{\delta^k}-Mt-\frac{b}{T-t}-\frac{b}{T-s}
		\end{equation}
		for $\delta\in(0,1)$, $(x,t),(y, s)\in K_T$.  For each $\delta$, we have a maximizer $(x_\delta,t_\delta,y_\delta,s_\delta)\in K_T\times K_T$.
		It follows that
		\begin{equation*}
			\Psi_\delta(x_\delta,t_\delta,y_\delta,s_\delta)\geq  \Psi(\hat{x}, \hat{t})
		\end{equation*}
		and thus
		\[
		\begin{aligned}
			\frac{|x_\delta-y_\delta|^k}{\delta^k}+\frac{|t_\delta-s_\delta|^k}{\delta^k}\leq\  & u_\vep(x_\delta, t_\delta)-u_0(y_\delta, s_\delta) -u_\vep(\hat{x}, \hat{t})+u_0(\hat{x}, \hat{t})\\ & -M(t_\delta-\hat{t})-\frac{b}{T-t_\delta}-\frac{b}{T-s_\delta}+\frac{2b}{T-\hat{t}}.
		\end{aligned}
		\]
		Consequently, we obtain $|x_\delta-t_\delta|\to 0$ and $|t_\delta-s_\delta|\to 0$ as $\delta\to 0$. 

		Since $\Psi_\delta(x_\delta, t_\delta, y_\delta, s_\delta)\geq \Psi(x_\delta, t_\delta)$, we also get
		\begin{equation}\label{holder app1}
			u_0(x_\delta, t_\delta)-u_0(y_\delta, s_\delta)\geq \frac{|x_\delta-y_\delta|^k}{\delta^k}+\frac{|t_\delta-s_\delta|^k}{\delta^k}.
		\end{equation}
		The H\"older continuity of $u_0$ in \eqref{equi-holder} yields 
		\begin{equation}\label{holder app2}
			u_0(x_\delta, t_\delta)-u_0(y_\delta, s_\delta)\leq L(|x_\delta-y_\delta|^\theta+|t_\delta-s_\delta|^\theta)
		\end{equation} 
		Combining \eqref{holder app1} and \eqref{holder app2}, we deduce that 
		\[
		\frac{|x_\delta-y_\delta|^k}{\delta^k}+\frac{|t_\delta-s_\delta|^k}{\delta^k}\leq  L(|x_\delta-y_\delta|^\theta+|t_\delta-s_\delta|^\theta),
		\]
		which implies that 
		\begin{equation}
			\label{lip-est1}
			|x_\delta-y_\delta|\leq 2L \delta^{\frac{k}{k-\theta}}, \quad |t_\delta-s_\delta|\leq 2L \delta^{\frac{k}{k-\theta}}. 
		\end{equation}

		\noindent {\bf Step 2. Application of the Crandall-Ishii Lemma.}
		Assume thriving for a contradiction that both $(x_\delta, t_\delta), (y_\delta, s_\delta)\in \Om_T$. Then we can use the $(n+1)$-dimensional elliptic-type Crandall-Ishii Lemma \cite[Theorem~3.2]{CIL} (or \cite[Lemma~3.5]{OS}) to obtain $(\tau, \xi, X)\in \ol{P}^{2, +}u_\vep(x_\delta, t_\delta)$, $(\sigma, \eta, Y)\in \ol{P}^{2, -}u(y_\delta, s_\delta)$ for each $\mu>0$ such that
		\begin{equation}\label{ishii-1}
			\xi=\eta=\frac{k}{\delta^k} |x_\delta-y_\delta|^{k-2}(x_\delta-y_\delta),
		\end{equation}   
		\begin{equation}
			\label{ishii-time}
			\tau=\frac{k(t_\delta-s_\delta)^{k-1}}{\delta^k}+M+\frac{b}{(T-t_\delta)^2}, \quad \sigma=\frac{k(t_\delta-s_\delta)^{k-1}}{\delta^k}-\frac{b}{(T-s_\delta)^2},
		\end{equation}
		and
		\begin{equation}\label{ishii-2}
			\begin{pmatrix}
				X & 0\\
				0 & -Y
			\end{pmatrix}\leq Z +\mu Z^2
		\end{equation}
		with
		\begin{equation}\label{ishii-z}
			Z= \frac{k(k-2)}{\delta^k} |x_\delta-y_\delta|^{k-4} \begin{pmatrix}
				(x_\delta-y_\delta)\otimes  (x_\delta-y_\delta) & - (x_\delta-y_\delta)\otimes  (x_\delta-y_\delta) \\
				- (x_\delta-y_\delta)\otimes  (x_\delta-y_\delta) &  (x_\delta-y_\delta)\otimes  (x_\delta-y_\delta)
			\end{pmatrix}.
		\end{equation}
		For the definition of parabolic semijets $\ol{P}^{2, +}$ and $\ol{P}^{2, -}$, we refer to Section 8 of \cite{CIL}. By direct computations, it is not difficult to see that 
		\[
		\begin{aligned}
			& Z+\mu Z^2\\
			&=\left(\frac{k(k-2)}{\delta^k} |x_\delta-y_\delta|^{k-4} +C_\mu\right)\begin{pmatrix}
				(x_\delta-y_\delta)\otimes  (x_\delta-y_\delta) & - (x_\delta-y_\delta)\otimes  (x_\delta-y_\delta) \\
				- (x_\delta-y_\delta)\otimes  (x_\delta-y_\delta) &  (x_\delta-y_\delta)\otimes  (x_\delta-y_\delta)
			\end{pmatrix},    
		\end{aligned}
		\]
		where 
		\begin{equation}\label{C-mu}
			C_\mu= \mu  \frac{k^2(k-2)^2}{\delta^{2k}} |x_\delta-y_\delta|^{2k-6}.
		\end{equation}
		
		\noindent {\bf Step 3. Estimates for the first order terms.}
		Let us next estimate the difference between  $H_\vep(x_\delta, t_\delta, \xi)$ and $H_0(y_\delta, s_\delta, \eta)$.
		Using (A3), we deduce 
		\[
		\begin{aligned}
			&H_\vep(x_\delta,t_\delta,  \xi)-H_0(y_\delta,s_\delta,  \eta)\\
			&\geq H_\vep(x_\delta,t_\delta,  \xi) -H_0(x_\delta,t_\delta,  \xi)-L_H(1+\abs{\xi})(\abs{x_\delta-y_\delta}+\abs{t_\delta-s_\delta}).
		\end{aligned}
		\]
		It then follows from (C2) that 
		\begin{equation}\label{h-est1}
			H_\vep(x_\delta, t_\delta,  \xi)-H_0(y_\delta, s_\delta,  \eta)\geq -c_H \vep^\gamma -L_H(1+\abs{\xi})(\abs{x_\delta-y_\delta}+\abs{t_\delta-s_\delta}).
		\end{equation}
		We can use \eqref{lip-est1} and \eqref{ishii-1} to estimate the last term
		\begin{align*}
			L_H(1+\abs{\xi})(\abs{x_\delta-y_\delta}+\abs{t_\delta-s_\delta})&=L_H\left(1+\frac{k}{\delta^k} |x_\delta-y_\delta|^{k-1}\right)(\abs{x_\delta-y_\delta}+\abs{t_\delta-s_\delta})\\
			&\leq L_H\left(1+(2L)^{k-1} k\delta^{\frac{k(k-1)}{k-\theta}-k}\right)\left(4L \delta^{\frac{k}{k-\theta}}\right) \\
			&=4L_HL\left(\delta^{\frac{k}{k-\theta}}+(2L)^{k-1}k\delta^{\frac{k\theta}{k-\theta}}\right).
		\end{align*}
		Taking the minimum between the two exponents for $\delta$, we are led to                      
		\begin{equation}\label{h-est2}
			H_\vep(x_\delta,t_\delta,  \xi)-H_0(y_\delta,s_\delta,  \eta) \geq -c_H \vep^\gamma -C_1 \delta^{\frac{k\theta}{k-\theta}},
		\end{equation}
		where we set
		\[
		C_1=4 LL_H+2^{k+1}L^{k}kL_H.
		\]
		\noindent {\bf Step 4. Estimates for the second order terms.} We consider two cases. 
		
		\noindent Case 1. $\xi=\eta\neq 0$. In this case, multiplying both sides of \eqref{ishii-2} from left and right by $A_\vep(\xi)^{1/2}\zeta\oplus A_0(\xi)^{1/2}\zeta\in \R^{2n}$ with an arbitrary $\zeta\in \R^n$, $|\zeta|=1$, we obtain 
		\[
		\begin{aligned}
			&\la  X A_\vep(\xi)^{1/2}\zeta, A_\vep(\xi)^{1/2}\zeta \ra -\la  Y A_0(\xi)^{1/2}\zeta, A_0(\xi)^{1/2}\zeta \ra \\
			\leq & \left(\frac{k(k-2)}{\delta^k} |x_\delta-y_\delta|^{k-4}  +C_\mu\right) \left|A_\vep(\xi)^{1/2}(x_\delta-y_\delta)-A_0(\xi)^{1/2}(x_\delta-y_\delta)\right|^2,   
		\end{aligned}
		\]
		where $C_\mu>0$ is given by \eqref{C-mu}.  
		Using the arbitrariness of $\zeta$, we thus have 
		\begin{equation}
			\begin{aligned}
				\label{a-est0}
				& \la X A_\vep(\xi)^{1/2},   A_\vep(\xi)^{1/2}\ra - \la Y A_0(\xi)^{1/2}, A_0(\xi)^{1/2}\ra\\
				& \leq \left(\frac{k(k-2)}{\delta^k} |x_\delta-y_\delta|^{k-4}  +C_\mu\right) \left\|A_\vep(\xi)^{1/2}-A_0(\xi)^{1/2}\right\|^2  I. 
			\end{aligned}
		\end{equation}
		By (C1), we have
		\begin{equation*}
			\norm{A_\vep(\xi)^{1/2} -A_0(\xi)^{1/2}}\leq c_A\vep^{\alpha}(1+|\xi|^\beta).
		\end{equation*}
		which implies that
		\begin{align*}
			\norm{A_\vep(\xi)^{1/2} -A_0(\xi)^{1/2}}^2
			&\leq c_A^2\vep^{2\alpha}(1+|\xi|^\beta)^2
			\leq 4c_A^2\eps^{2\alpha}(1+\abs{\xi}^{2\beta}). 
		\end{align*}
		Applying this estimate to \eqref{a-est0}, together with the value of $\xi$, we have
		\[
		\begin{aligned}
			&\tr (A_\vep(\xi)X) -\tr (A_0(\xi)Y)\\
			&\leq \ 4c_A^2 k(k-2)n\delta^{-k}|x_\delta-y_\delta|^{k-2} \vep^{2\alpha}+ 4c_A^2 k^{1+2\beta}(k-2)n\delta^{-k(1+2\beta)} |x_\delta-y_\delta|^{k-2+2(k-1)\beta}\vep^{2\alpha} \\
			&\quad\quad +n C_\mu 4c_A^2\eps^{2\alpha}(1+\abs{\xi}^{2\beta}). 
		\end{aligned}
		\]
		Note that in the estimate above $k-2+2(k-1)\beta>0$ due to the condition $\beta> \frac{2-k}{2(k-1)}$ in (C1).
		
		Adopting \eqref{lip-est1} again, we are led to 
		\begin{equation}\label{a-est1}
			\tr (A_\vep(\xi)X) -\tr (A_0(\xi)Y)\leq C_2\vep^{2\alpha}\left(\delta^{\frac{k(\theta-2)}{k-\theta}}+ \delta^{\frac{k(\theta-1)(1+2\beta)-k}{k-\theta}}\right)+n C_\mu 4c_A^2\eps^{2\alpha}(1+\abs{\xi}^{2\beta}), 
		\end{equation}
		where 
		\[
		C_2=4c_A^2 k(k-2)n (2L)^{k-2}+4c_A^2 k^{1+2\beta}(k-2)n (2L)^{k-2+2(k-1)\beta}.
		\]
		The definition of subsolutions and supersolutions enables us to obtain
		\[
		\tau-\tr (A_\vep(\xi)X)+H_\vep(x_\delta, t_\delta,  \xi)\leq 0
		\]
		and
		\[
		\sigma-\tr (A_0(\eta)Y)+H_0(y_\delta, s_\delta,  \eta)\geq 0.
		\]
		By taking the difference of these inequalities and applying \eqref{h-est2} and \eqref{a-est1}, we deduce 
		\begin{align*}
			&M+\frac{b}{(T-t_\delta)^2}+\frac{b}{(T-s_\delta)^2}=\tau-\sigma\\
			&\leq\tr (A_\vep(\xi)X)-\tr (A_0(\eta)Y)-\left(H_\vep(x_\delta,t_\delta,\xi)-H_0(y_\delta,s_\delta,\eta)\right) \\
			&\leq M_{\vep, \delta}+n C_\mu 4c_A^2\eps^{2\alpha}(1+\abs{\xi}^{2\beta})
		\end{align*}
		for
		\[
		M_{\eps,\delta}=C_2\vep^{2\alpha}\left(\delta^{\frac{k(\theta-2)}{k-\theta}}+\delta^{\frac{k(\theta-1)(1+2\beta)-k}{k-\theta}}\right)+c_H \vep^\gamma +C_1 \delta^{\frac{k\theta}{k-\theta}} 
		\]
		Our estimate no longer depends on the matrices $X$ and $Y$, and thus we can let $\mu\to 0$ to obtain 
		\[
		M+\frac{b}{(T-t_\delta)^2}+\frac{b}{(T-s_\delta)^2}\leq M_{\vep, \delta}. 
		\]
		\noindent Case 2. $\xi=\eta=0$. For this case, we apply the definition of $\F$-subsolutions and $\F$-supersolutions to get
		\[
		H_\vep(x_\delta,t_\delta,  \xi)-H_0(x_\delta,t_\delta,  \xi)\leq -\tau+\sigma\leq -M-\frac{b}{(T-t_\delta)^2}-\frac{b}{(T-s_\delta)^2}  ,
		\]
		which by \eqref{h-est2} yields 
		\[
		M+\frac{b}{(T-t_\delta)^2}+\frac{b}{(T-s_\delta)^2}\leq c_H\vep^\gamma +C_1 \delta^{\frac{k\theta}{k-\theta}}<M_{\eps,\delta}
		\]
		for all $\delta>0$ small. Thus for any $\xi=\eta$, we have the same inequality
		\begin{equation*}
			M+\frac{b}{(T-t_\delta)^2}+\frac{b}{(T-s_\delta)^2}\leq M_{\eps,\delta},	
		\end{equation*}
		which means that for $(x_\delta, t_\delta)$ and $(y_\delta, s_\delta)$ cannot be both in $\Om_T$ if we choose $M=M_{\vep, \delta}$ at the beginning of the proof.
		
		\noindent {\bf Step 5. Conclusion via boundary data.} 
		Our preceding arguments show that either $(x_\delta, t_\delta)\in \partial_p \Om_T$ or $(y_\delta, s_\delta)\in \partial_p \Om_T$. Then,  using the equi-H\"older continuity of $u_\vep$ in \eqref{equi-holder}, we have  
		\begin{equation}\label{bdry-max para}
			\begin{aligned}
				u_\vep(x_\delta, t_\delta)-u_0(y_\delta, s_\delta)&\leq \sup_{\partial_p \Om_T} |g_\vep-g| +L(|x_\delta-y_\delta|^\theta+|t_\delta-s_\delta|^\theta)\\
				&\leq \sup_{\partial_p \Om_T} |g_\vep-g|+4L^2  \delta^{\frac{k\theta}{k-\theta}}. 
			\end{aligned}
		\end{equation}
		Hence, we obtain
		\[
		\begin{aligned}
			u_\vep(x, t)-u_0(x, t)-M_{\vep, \delta} t-\frac{b}{T-t}
			&\leq u_\vep(x_\delta, t_\delta)-u(y_\delta, s_\delta)-M_{\vep, \delta} t_\delta-\frac{b}{T-t_\delta}\\
			&\leq \sup_{\partial_p \Om_T} |g_\vep-g_0| +4L^2  \delta^{\frac{k\theta}{k-\theta}}
		\end{aligned}
		\]
		for all $(x, t)\in \Om_T$. Our estimate no longer depends on the maximizer, and thus we can let $b\to 0$. For all $t\in[0,T)$, we have  
		\begin{align*}
			\label{eq:almostfinal}
			&\max_{x\in \Oba} (u_\vep(x, t)-u_0(x, t))\leq \sup_{\partial_p \Om_T} |g_\vep-g_0| +4L^2  \delta^{\frac{k\theta}{k-\theta}} +M_{\vep, \delta}t \\
			& = \sup_{\partial_p \Om_T} |g_\vep-g_0| +(4L^2+C_1 t)  \delta^{\frac{k\theta}{k-\theta}} +C_2 \vep^{2\alpha}t\left(\delta^{\frac{k(\theta-2)}{k-\theta}}+ \delta^{\frac{k(\theta-1)(1+2\beta)-k}{k-\theta}}\right) +c_H \vep^\gamma t. \numberthis
		\end{align*}
		To finish our proof, we need to optimize the right hand side over $\delta$. Denote
		\begin{equation*}
			a=\frac{k\theta}{k-\theta}>0, \quad b=\frac{k(\theta-2)}{k-\theta}<0, \quad c=\frac{k(\theta-1)(1+2\beta)-k}{k-\theta}<0,
		\end{equation*}
		and as $\delta\in(0,1)$, we have $\delta^b+\delta^c\leq 2\delta^d$
		for $d=\min\{b,c\}<0$. It essentially suffices to minimize 
		\begin{equation*}
			f(\delta)= h_1\delta^a+2h_2\eps^{2\alpha}\delta^d
		\end{equation*}
		with the constants $h_1=4L^2+C_1 t$, $h_2=C_2 \vep^{2\alpha}t$. By differentiating $f$, we see that $f(\delta)$ achieves a strict global minimum at $\delta=\delta_0$, given by 
		\begin{equation*}
			\delta_0=\left(\frac{-2h_2d}{ah_1}\right)^{\frac{1}{a-d}}\eps^{\frac{2\alpha}{a-d}}.
		\end{equation*}
		For $\eps>0$ small enough, we have $\delta_0\in(0,1)$, because $\frac{2\alpha}{a-d}>0$. At the minimizer, 
		\begin{equation*}
			h_1a\delta_0^{a}=-2h_2d\eps^{2\alpha}\delta_0^{d}
		\end{equation*}
		holds and thus
		\begin{equation*}
			h_1\delta_0^a+h_2\eps^{2\alpha}\left(\delta_0^b+\delta_0^c\right)\leq \left(1-\frac{a}{d}\right)h_1\delta_0^a\leq C\eps^{\frac{2\alpha a}{a-d}},
		\end{equation*}
		for $C>0$ independent of $\vep$. The exponent can be simplified based on the value of $\beta$. If $\beta\leq 0$, then $d=b$,
		\begin{equation*}
			a-d=\frac{2k}{k-\theta} \quad \text{ and }\quad  \frac{2\alpha a}{a-d}=\alpha\theta,
		\end{equation*}
		and if $\beta\geq0$, then $d=c$,
		\begin{equation*}
			a-d=\frac{2k(1+\beta(1-\theta))}{k-\theta} \quad \text{ and }\quad  \frac{2\alpha a}{a-d}=\frac{\alpha\theta}{1+\beta(1-\theta)}.
		\end{equation*}

		Using this estimate to optimize over $\delta$ in \eqref{eq:almostfinal}, we are led to the desired estimate
		\[
		\sup_{K_T} (u_\vep-u_0)\leq \sup_{\partial_p \Om_T} |g_\vep-g_0|+C\vep^{\nu}  +c_HT\vep^\gamma
		\]
		for $\nu>0$ given by \eqref{nu-exponent} and some $C>0$. 
		We conclude the proof, noticing that the same estimate for the upper bound of $u_0-u_\vep$ can be proved similarly. 
	\end{proof}
	
	\begin{rmk}
		Our assumptions (C1)(C2) only serve as a convenient prototype for describing the convergence of the operators. These conditions can be generalized by replacing $c_A \vep^\alpha$ and $c_H\vep^\gamma$ with general moduli of continuity $\omega_A(\vep)$ and $\omega_H(\vep)$ respectively. In this case, the proof can be adapted to yield a more general convergence rate 
		\[
		\sup_{\ol{\Omega_T}}|u_\vep-u_0|\leq \sup_{\partial_p \Om_T}|g_\vep-g_0|+C\omega_A(\vep)^{\nu'}  +T\omega_H(\vep) 
		\]
		for some $C>0$ and
		\[
		\nu'=\frac{\theta}{1+(1-\theta)\max\{\beta, 0\}}.
		\]
		It is also possible to substitute the term $|\xi|^\beta$ in \eqref{eq:c1} with other functions of $|\xi|$. However, in this work we restrict our attention to the typical setting and do not intend to treat the result in full generality. 
	\end{rmk}
	
	\begin{rmk}
		Under the same assumptions on the operators, our result immediately implies the uniqueness of H\"older continuous solutions to \eqref{general parabolic lim eq}
		with initial boundary data \eqref{parabolic bdry-cond} by taking $A_\vep=A$, $H_\vep=H$ and $g_\vep=g$ for all $\vep\geq 0$. This is naturally expected, as our proof is based on the comparison arguments.  
	\end{rmk}
	
	\begin{rmk}\label{rmk holder}
		The H\"older continuity in time is used in our proof to handle possible strong singularity of the operator. It can be dropped if $A_\vep: \R^n_0\to \S^n_+$ is bounded for each $\vep\geq 0$. In this case, we only need to apply the doubling variable technique to space variables, considering the maximum of 
		\[
		\Psi_\delta(x, y, t)=u_\eps(x,t)-u_0(x, t)-\frac{\abs{x-y}^k}{\delta^k}-Mt-\frac{b}{T-t}
		\]
		(instead of $\Psi_\delta$ as in \eqref{double variable1}) over  $\Oba\times \Oba\times [0, T)$. We now essentially have a maximizer satisfying $t_\delta=s_\delta$, and thus can still derive \eqref{lip-est1} by using only the H\"older continuity of $u_0$ in space together with the relation $\Psi_\delta(x_\delta, y_\delta, t_\delta)\geq \Psi(x_\delta, t_\delta)$. The boundedness of $A_\vep$ enables us to adopt the parabolic version of Crandall-Ishii Lemma as in \cite[Theorem~8.3]{CIL}, and all of the estimates about semijets in our proof still hold. Avoiding doubling time variables also yields \ref{bdry-max para} with $t_\delta=s_\delta$ under equi-H\"older continuity of $u_\vep$ only in space. With further details omitted, we state the quantitative stability result for such parabolic equations as below. 
	\end{rmk}
	\newpage
	\begin{thm}\label{thm:parabolic2}
		Let $T>0$, $\Omega$,  $A_\vep$, $H_\vep$ and $g_\vep$ satisfy all the assumptions in Theorem \ref{thm:parabolic}. Assume also that $A_\vep$ is bounded in $\R^n_0$ for each $\vep>0$. Let $u_\vep\in C(\ol{\Om_T})$ be bounded $\F$-solutions of  \eqref{general parabolic approx eq}\eqref{parabolic approx bdry-cond}. Assume in addition that $u_\vep(\cdot, t)$ are equi-H\"older continuous in $\overline{\Omega}$, that is, there exist $L>0$ and $\theta\in (0, 1]$ such that,   
		\begin{equation}\label{equi-holder-x}
			|u_\vep(x_1, t)-u_\vep(x_2, t)|\leq L|x_1-x_2|^\theta
		\end{equation}
		for all $x_1, x_2\in \ol{\Omega}$, $t\in [0, T)$ and $\vep\geq 0$ small. Then, \eqref{error-est parabolic} holds for all $\vep>0$ small with $\nu$ given as in \eqref{nu-exponent}. 
	\end{thm}
	
	\begin{rmk}\label{rmk:periodic}
		Our discussions on the Cauchy-Dirichlet problem in Theorem~\ref{thm:parabolic} and Theorem~\ref{thm:parabolic2} extend readily to the case of spatially periodic settings. 
		Instead of considering \eqref{general parabolic lim eq} in a bounded domain $\Omega$, we may obtain the same stability results as in Theorem~\ref{thm:parabolic} and Theorem~\ref{thm:parabolic2} for the whole space $\R^n$, provided that the initial value $g$ and $H$ are both periodic with respect to $x$. Note that the periodic setting essentially plays the same role as a bounded domain (without boundary) in our comparison arguments and therefore does not substantially affect the proofs. For clarity of presentation, here we choose not to repeat the statements for this case.
	\end{rmk}
	
	The proofs of our preceding parabolic results can be adapted to quasilinear elliptic equations of the same type with slightly altered assumptions. We include the dependence of $H$ on $u$ and assume a strict monotonicity on the $u$-dependence. 
	
	Let $\Omega$ be a bounded domain in $\R^n$. For $\vep\geq 0$, consider \eqref{general approx eq} with the boundary condition \eqref{approx bdry-cond}.
	Our assumptions on $A_\vep$ and $H_\vep$ in the elliptic case include (A1) and the following (A2') (in place of (A2)) and (A3).
	\begin{enumerate}
		\item[(A2')] $H_\vep: \Omega\times \R\times \R^n\to \R$ is continuous, and there exists $L_H>0$  such that
		\[
		\begin{aligned}
			|H_\vep(x_1, \rho, \xi)-H_\vep(x_2, \rho, \xi)|&\leq L_H(1+|\xi|)\left(|x_1-x_2|\right)\\
		\end{aligned}
		\]
		for all $x_1, x_2\in \Omega$, $\rho\in \R$, $\xi\in \R^n$, and all $\vep\geq 0$.
		\item[(A3)] There exists $\lambda>0$ such that $H_\vep: \Omega\times \R\times \R^n\to \R$ satisfies
		\begin{equation}\label{h-monotone-strict}
			H_\vep(x, \rho_1, \xi)-H_\vep(x, \rho_2, \xi) \geq \lambda (\rho_1-\rho_2)
		\end{equation}
		for all $x\in \Omega$, $\xi\in \R^n$, $\rho_1\geq \rho_2$, and all $\vep\geq 0$. 
	\end{enumerate}
	Moreover, we substitute the assumption (C2) with the following (C2'):
	\begin{enumerate}
		\item[(C2')] For any $R>0$, there exist $c_H>0$ and $\gamma>0$  such that 
		\[
		|H_\vep(x, r, \xi)-H_0(x, r, \xi)|\leq c_H \vep^\gamma 
		\] 
		holds for all $\vep>0$ small, $x\in \Omega, \xi\in \R^n$ and $r\in \R$ with $|r|<R$.
	\end{enumerate}
	
	
	Our quantitative stability result for the elliptic equation is as below. For the sake of completeness, we also include a proof, which to a large extent resembles that of Theorem~\ref{thm:parabolic}. 
	\newpage
	\begin{thm}[Quantitative stability for elliptic equations]\label{thm:elliptic}
		Let $\Omega$ be a bounded domain in $\R^n$. For $\vep\geq 0$, assume that $A_\vep: \R^n_0\to \S_+^n$ and $H_\vep: \Omega\times \R\times  \R^n\to \R$ are continuous functions satisfying (A1)(A2')(A3) and (C1)(C2'). Let $g_\vep\in C(\partial\Omega)$ and $u_\vep\in C(\Oba)$ be $\F$-solutions of \eqref{general approx eq}\eqref{approx bdry-cond}. Assume in addition that $u_\vep$ are equi-H\"older continuous in $\Oba$, that is, there exist $L>0$ and $0<\theta\leq 1$ such that  
		\[
		|u_\vep(x_1)-u_\vep(x_2)|\leq L|x_1-x_2|^\theta 
		\]
		for all $x_1, x_2\in \Oba$ and $\vep\geq 0$.
		Then, for all $\vep>0$ small,  
		\begin{equation}\label{error-est}
			\max_{\Oba} |u_\vep-u_0|\leq \max_{\partial \Omega} |g_\vep-g|+\frac{1}{\lambda}\left(C_0\vep^{\nu}+c_H \vep^\gamma \right)
		\end{equation}
		holds with $\nu$ given as in \eqref{nu-exponent}, where $C_0>0$ is a constant depending on 
		$k>2$,$L_A>0$, $L_H>0$, $\lambda>0$ in (A1)(A2')(A3) and $\alpha, \beta, c_A, \in \R$ given in (C1). The constants $\gamma, c_H>0$ are given as in (C2'). 
	\end{thm}

	\begin{proof}
		The proof is similar to the parabolic case, so we will skip some details and point out the differences. Suppose that there exists $\hat{x}\in \Oba$ such that 
		\[
		u_\vep(\hat{x}) -u_0(\hat{x})=\max_{\Oba} (u_\vep-u_0).
		\]
		For $\delta\in (0, 1)$ small, consider
		\[
		\Phi_\delta(x, y):= u_\vep(x)-u_0(y)-\frac{|x-y|^k}{\delta^k}, \quad x, y\in \Oba. 
		\] 
		It is clear that there exists a maximizer $(x_\delta, y_\delta)\in \Oba\times \Oba$ of $\Phi_\delta$, which yields
		\begin{equation}\label{max-est2}
			\Phi_\delta(x_\delta, y_\delta)=u_\vep(x_\delta)-u_0(y_\delta)-\frac{|x_\delta-y_\delta|^k}{\delta^k}\geq \Phi_\delta(\hat{x}, \hat{x})=\max_{\Oba} (u_\vep-u). 
		\end{equation}
		We thus get 
		\[
		\frac{|x_\delta-y_\delta|^k}{\delta^k}\leq u_\vep(x_\delta) -u(y_\delta)-u_\vep(\hat{x})+u(\hat{x}), 
		\]
		which, by the H\"older continuity of $u$ in $\Oba$ yields,  
		\[
		\frac{|x_\delta-y_\delta|^k}{\delta^k}\leq u_\vep(x_\delta)-u_0(x_\delta)-u_\vep(\hat{x})+u_0(\hat{x})+L|x_\delta-y_\delta|^\theta\leq L|x_\delta-y_\delta|^\theta
		\]
		and thus
		\begin{equation}\label{lip-est0}
			|x_\delta-y_\delta|\leq L^{\frac{1}{k-\theta}}\delta^{\frac{k}{k-\theta}}. 
		\end{equation}
		The following argument is essentially the same as in the elliptic case. We treat separately the cases where either of $x_\delta$ and $y_\delta$ lies on $\partial \Omega$ and where $x_\delta, y_\delta\in \Omega$.
		
		Let us first consider the case when either $x_\delta\in \partial \Omega$ or $y_\delta\in \partial \Omega$ holds. Using the equi-H\"older continuity of $u_\vep$,  we have 
		\begin{equation}\label{bdry-max}
			u_\vep(x_\delta)-u_0(y_\delta)\leq \max_{\partial \Omega} |u_\vep-u_0|+L|x_\delta-y_\delta|^\theta\leq \max_{\partial \Omega} |g_\vep-g|+  L^{\frac{k}{k-\theta}}\delta^{\frac{k\theta}{k-\theta}}. 
		\end{equation}

		
		Suppose now that $x_\delta, y_\delta\in \Omega$. We apply the Crandall-Ishii lemma \cite{CIL} to get, for every $\mu>0$ small, $(\xi, X)\in \ol{J}^{2, +}u_\vep(x_\delta)$ and $(\eta, Y)\in \ol{J}^{2, -}u(y_\delta)$ satisfying the same conditions as in \eqref{ishii-1} and \eqref{ishii-2}.

		By using (A2'), (A3), \eqref{ishii-1} and (C1)(C2'), we obtain 
		\begin{equation}\label{elliptic h-est1}
			\begin{aligned}
				& H_\vep(x_\delta, u_\vep(x_\delta), \xi)-H(y_\delta, u(y_\delta), \eta)\\
				&\geq -c_H \vep^\gamma -L_H|x_\delta-y_\delta|-\frac{L_H k}{\delta^k}|x_\delta-y_\delta|^k+\lambda (u_\vep(x_\delta)-u(y_\delta)).
			\end{aligned}
		\end{equation}
		for $\lambda>0$ given in (A3).
		By \eqref{lip-est0} and taking the smaller of the two exponents for $\delta$, we are led to                      
		\begin{equation}\label{h-est2b}
			H_\vep(x_\delta, u_\vep(x_\delta), \xi)-H(y_\delta, u(y_\delta), \eta) \geq -c_H \vep^\gamma -C_1 \delta^{\frac{k\theta}{k-\theta}} +\lambda(u_\vep(x_\delta)-u(y_\delta))
		\end{equation}
		for some $C_1>0$.
		
		Our estimates for the second order terms are similar to the parabolic case. 
		
		\noindent Case 1. $\xi=\eta\neq 0$. Multiplying both sides of \eqref{ishii-2} from left and right by $A_\vep(\xi)^{1/2}\zeta\oplus A(\xi)^{1/2}\zeta\in \R^{2n}$ and using (C1) yields 
		\[
		\begin{aligned}
			&\tr (A_\vep(\xi)X) -\tr (A(\xi)Y)\\
			&\leq c_A^2 k(k-2)n\delta^{-k}|x_\delta-y_\delta|^{k-2} \vep^{2\alpha}+ c_A^2 k^{1+2\beta}(k-2)n\delta^{-k(1+2\beta)} |x_\delta-y_\delta|^{k-2+2(k-1)\beta}\vep^{2\alpha} \\
			&\quad \quad +nC_\mu4c_A^2\eps^{2\alpha}(1+\abs{\xi}^{2\beta}),
		\end{aligned}
		\]
		where the condition $\beta> \frac{2-k}{2(k-1)}$ in (C1) immediately yields $k-2+2(k-1)\beta>0$. 
		
		Adopting \eqref{lip-est0} again, we are led to 
		\begin{align*}\label{a-est1b}
			&\tr (A_\vep(\xi)X) -\tr (A(\xi)Y)\\&\leq C_2\vep^{2\alpha}\left(\delta^{\frac{k(\theta-2)}{k-\theta}}+ \delta^{\frac{k(\theta-1)(1+2\beta)-k}{k-\theta}}\right)+nC_\mu4c_A^2\eps^{2\alpha}(1+\abs{\xi}^{2\beta})\numberthis
		\end{align*}
		for the same constant $C_2>0$ as in the parabolic case.
		Applying the definition of subsolutions and supersolutions, we get in this case
		\[
		-\tr (A_\vep(\xi)X)+H_\vep(x_\delta, u_\vep(x_\delta), \xi)\leq 0
		\]
		and
		\[
		-\tr (A(\xi)Y)+H(y_\delta, u(y_\delta), \xi)\geq 0.
		\]
		Taking the difference of these inequalities and adopting \eqref{h-est2b} and \eqref{a-est1b}, we are led to an estimate independent of the matrices $X$ and $Y$. Sending $\mu\to0$ yields 
		\[
		\lambda(u_\vep(x_\delta)-u(y_\delta)) \leq c_H \vep^\gamma +C_1 \delta^{\frac{k\theta}{k-\theta}}+C_2\vep^{2\alpha}\left(\delta^{\frac{k(\theta-2)}{k-\theta}}+ \delta^{\frac{k(\theta-1)(1+2\beta)-k}{k-\theta}}\right). 
		\]

		\noindent Case 2. $\xi=\eta=0$. The definition of $\F$-subsolutions and $\F$-supersolutions give us
		\[
		H_\vep(x_\delta, u_\vep(x_\delta), \xi)\leq 0\leq H(x_\delta, u(y_\delta), \xi),
		\]
		which by \eqref{h-est2b} implies that for all $\delta>0$ small,  
		\[
		\lambda(u_\vep(x_\delta)-u(y_\delta))\leq c_H\vep^\gamma +C_1 \delta^{\frac{k\theta}{k-\theta}}.
		\]

		Combining the estimates in both Case 1 and Case 2, we obtain 
		\[
		u_\vep(x_\delta)-u(y_\delta)\leq \frac{c_H}{\lambda} \vep^\gamma + \frac{C_1}{\lambda} \delta^{\frac{k\theta}{k-\theta}}+\frac{C_2}{\lambda} \vep^{2\alpha}\left( \delta^{\frac{k(\theta-2)}{k-\theta}}+\delta^{\frac{k(\theta-1)(1+2\beta)-k}{k-\theta}}\right)
		\]
		when $x_\delta, y_\delta\in \Omega$. By including the estimate \eqref{bdry-max} for the case when $x_\delta \in \partial \Omega$ or $y_\delta \in \partial \Omega$, we are led to 
		\[
		\begin{aligned}
			u_\vep(x_\delta)-u(y_\delta)\leq & \max_{\Oba}  |g_\vep-g|+\frac{c_H}{\lambda} \vep^\gamma \\
			& + \left(\frac{C_1}{\lambda}+L^{\frac{k}{k-\theta}} \right) \delta^{\frac{k\theta}{k-\theta}}+\frac{C_2}{\lambda} \vep^{2\alpha}\left(\delta^{\frac{k(\theta-2)}{k-\theta}}+ \delta^{\frac{k(\theta-1)(1+2\beta)-k}{k-\theta}}\right)
		\end{aligned}
		\]
		for all $\vep>0$ small and $\delta\in (0, 1)$. 
		Minimizing the right hand side over $\delta$, we obtain 
		\[
		u_\vep(x_\delta)-u(y_\delta)\leq \max_{\Oba} |g_\vep-g|+\frac{c_H}{\lambda} \vep^\gamma +  \frac{C_0}{\lambda} \vep^\nu
	\]
	for $\nu$ as in \eqref{nu-exponent} and for some $C_0>0$ depending on $C_1, C_2, L, k, \theta, \lambda$, 
	which completes the proof of \eqref{error-est1}. The estimate for $\max_{\Oba} (u-u_\vep)$ is proven similarly.
\end{proof}
\begin{rmk}
	The condition $\lambda>0$ in (A3) is needed for our method. It is not clear to us how to handle the case when $H$ does not depend on $u$. In particular, it would be interesting to establish an estimate of the difference between a $p$-harmonic function and a $q$-harmonic function for $p, q>1$ with common boundary data. 
\end{rmk} 

\section{Applications and examples}\label{sec:app}

We next discuss several typical examples to which we can apply Theorem~\ref{thm:parabolic} or Theorem~\ref{thm:parabolic2}.     
While it is also possible to obtain similar convergence rates for quasilinear elliptic equations of the same type by using Theorem~\ref{thm:elliptic}, we here focus only on parabolic equations and omit details for the elliptic case. 

\subsection{The normalized $p$-Laplace type equations}

Our first application is for 
\begin{equation}\label{np-parabolic}
	\partial_t u-\Delta^N_p u+H(x, u, \nabla u)=0 \quad \text{in $\Om_T:=\Omega\times (0, T)$}
\end{equation}  
with $1\leq p<\infty$ in a bounded domain $\Omega\subset \R^n$. As mentioned in the introduction, we are interested in the stability of the solution with respect to $p$. 

Fix $q\in (1, \infty)$. Let $\vep=|p-q|$ for $p\in (1, \infty)$. Suppose that $H_\vep$ does not depend on $\vep\geq 0$ small, that is,   
\[
H_\vep(x, \rho, \xi)=H_0(x, \rho, \xi)
\]
for all $x\in \Omega,\ \rho\in \R,\ \xi\in \R^n$. Let
\[
A_\vep(\xi)= I+(p-2) \frac{\xi\otimes \xi}{|\xi|^2}
\]
for all $\xi\in \R^n_0$ and $\vep\geq 0$. Since $A_\vep$ is bounded in $\R^n_0$, in this case we can apply Theorem~\ref{thm:parabolic2} to obtain the following result.

\begin{thm}\label{thm:p-lap1par}
	Let $\Omega\subset \R^n$ be a bounded domain and $T>0$. Assume that $H: \Omega\times (0,T)\times \R^n\to \R$ satisfies (A2).  Fix $q\in [1, \infty)$ and take $p\in [1, \infty)$ close to $q$. For any such $p$, let $u_p\in C(\Om_T)$ be a solution to \eqref{np-parabolic} satisfying \eqref{parabolic bdry-cond} with 
	with boundary data $g=g_p\in C(\partial_p \Om_T)$.  Assume in addition that there exist $L>0$ and $\theta\in (0, 1]$ such that
	\[
	|u_p(x_1, t)-u_p(x_2, t)|\leq L|x_1-x_2|^\theta 
	\]
	for all $x_1, x_2\in \overline{\Omega}$, $0\leq t<T$ and all $p$ near $q$. Then, there exists $C>0$ independent of $p$, such that
	\begin{equation}\label{rate:p-lap1par}
		\sup_{\overline{\Omega}\times[0,T)}|u_p-u_q|\leq \sup_{\partial_p \Om_T}|g_p-g_q| +C|p-q|^{\theta}. 
	\end{equation}
\end{thm}
\begin{proof}
	Note that for the normalized $p$-Laplacian,  (A1) holds with $k>0$ large for all $q$ near $p$. 
	Moreover, by a straightforward computation, we have
	\[
	A_\eps(\xi)^{1/2}=I+(\sqrt{p-1}-1) \frac{\xi\otimes \xi}{|\xi|^{2}}, \quad A_0(\xi)^{1/2}=I+(\sqrt{q-1}-1) \frac{\xi\otimes \xi}{|\xi|^{2}}
	\]
	with $p, q\in (1, \infty)$ satisfying $\vep=|p-q|$. 
	This yields the existence of $c>0$ such that
	\[
	\norm{A_\vep(\xi)^{1/2} -A_0(\xi)^{1/2}}\leq |\sqrt{p-1}-\sqrt{q-1}|\leq c|p-q|
	\]
	for all $q$ near $p$. This amounts to saying that (C1) holds with $\alpha=1$ and $\beta=0$.  We thus complete the proof of inequality \eqref{rate:p-lap1par} by applying Theorem~\ref{thm:parabolic2} with $c_H=0$ and $\nu=\theta$. 
\end{proof}

The result above implies the convergence rate $O(|p-q|)$ for spatially Lipschitz solutions $u_p\to u_q$ as $p\to q$ if they satisfy the same boundary value $g_p=g_q$ on $\partial \Omega_T$. A concrete example is as follows. 

\begin{example}
	Consider the normalized $p$-Laplace equation
	\begin{equation}\label{eq:np}
		\partial_t u=\Delta_p^N u \quad \text{in $\R^n\times (0, T)$}
	\end{equation}
	for $1\leq p<\infty$, $T>0$.  Assume that $u(x, t)$ is independent of the variables $x_2, \ldots, x_n$. Then the equation reduces to a one dimensional heat diffusion:
	\[
	\partial_t v- (p-1) \partial_{xx} v=0 \quad \text{in $\R\times (0, T)$.}
	\]
	Suppose that the initial value for this diffusion equation is given by 
	\[
	v(x, 0)=\sin x  \quad  x\in \R. 
	\]
	By periodicity, this initial value problem admits a unique solution, which is explicitly given by 
	\begin{equation}\label{eq:ex-np1}
		v(x, t)=e^{(1-p)t} \sin x, \quad x\in \R,\ t\geq 0. 
	\end{equation}
	It corresponds to the unique $\R^n$-periodic solution of the original Cauchy problem for \eqref{eq:np}. 
	
	One easily verifies that there exists $C>0$ such that the solution $v_p$ associated to the parameter $p\in [1, \infty)$ satisfies
	\[
	\sup_{\R^n\times [0, T)}|v_p-v_q|\leq C|p-q|
	\]
	for $p$ sufficiently close to any fixed $q\geq 1$. This estimate is consistent with our result in Theorem~\ref{thm:p-lap1par} for equi-Lipschitz solutions of \eqref{eq:np}, when adapted to the spatially periodic setting as discussed in Remark~\ref{rmk:periodic}. 
	
	This example also demonstrates the optimality of the convergence rate for Lipschitz solutions. However, if the initial data are not assumed to be Lipschitz, the rate provided by Theorem~\ref{thm:parabolic} may not be considered sharp. Indeed, when $p>1$ and the initial datum $v(\cdot,0)$ in this example is merely H\"older continuous in $\R$ with exponent $\theta\in(0,1)$, a direct application of Theorem~\ref{thm:p-lap1par} to spatially periodic solutions only yields a convergence rate 
	\[
	\sup_{\R\times [0, T)}|v_p- v_q|=O(|p-q|^\theta) \quad \text{as $p\to q$,}
	\]
	since Theorem~\ref{thm:p-lap1par} requires spatial equi-Hölder continuity of solutions on the entire time interval $[0,T)$. On the other hand, the regularizing effect of the heat equation actually implies that $v(\cdot,t)$ becomes Lipschitz continuous for $t>0$ arbitrarily small. As a result, if we use Theorem~\ref{thm:p-lap1par} on the truncated time interval $(\vep, T)$ for an arbitrarily small initial time $\vep>0$, then for each $t>0$, we recover the sharp rate 
	\[
	\sup_{\R}|v_p(\cdot, t)- v_q(\cdot, t)|=O(|p-q|).
	\]
	This highlights a limitation of our method when applied to general, possibly degenerate parabolic equations, since it does not take into account the regularizing effects of the evolution.  
\end{example}

\subsection{The variational $p$-Laplace type equations}
We now turn to the case of the variational $p$-Laplacian with $1<p<\infty$. Consider
\begin{equation}\label{vp-parabolic}
	\partial_t u-\Delta_p u+H(x, u, \nabla u)=0 \quad \text{in $\Om_T$,}
\end{equation}
for which we take 
\[
A_\eps(\xi)=|\xi|^{p-2}\left(I+(p-2) \frac{\xi\otimes \xi}{|\xi|^2}\right), \quad A_0(\xi)=|\xi|^{q-2}\left( I+(q-2) \frac{\xi\otimes \xi}{|\xi|^2}\right).
\]
We then have 
\[
\begin{aligned}
	&A_\eps(\xi)^{1/2}=|\xi|^{\frac{p}{2}-1}\left(I+(\sqrt{p-1}-1) \frac{\xi\otimes \xi}{|\xi|^{2}}\right),\\
	&A_0(\xi)^{1/2}=|\xi|^{\frac{q}{2}-1}\left(I+(\sqrt{q-1}-1) \frac{\xi\otimes \xi}{|\xi|^{2}}\right).     
\end{aligned}
\]
Below let us discuss several different cases in terms of the values of $q$. 

Suppose that $p<2$ and $q\leq 2$. Then for any $-1/2<\beta<q/2-1$, there exists $c>0$ such that  
\begin{equation}\label{A-est-vp}
	\norm{A_\vep(\xi)^{1/2} -A_0(\xi)^{1/2}}\leq c|p-q| (1+ |\xi|^{\beta})    
\end{equation}
for all $\xi\neq 0$ and $p\in (1, 2)$ close to $q$. 
By choosing $k>0$ large such that $\beta>\frac{2-k}{2k-1}$, we thus obtain (C1) with $\alpha=1$ and any $-1/2<\beta<p/2-1$. 

On the other hand, if $p>2$ and thus $q\geq 2$, then for any $\beta>q/2-1$, there exists $c>0$ such that \eqref{A-est-vp} holds for all $\xi\in \R^n$ and $p>2$ close to $q$. Thus, Theorem~\ref{thm:parabolic} implies the following.

\begin{thm}\label{thm:p-lap2par}
	Let $\Omega\subset \R^n$ be a bounded domain and $T>0$. Assume that $H: \Omega\times (0,T)\times \R^n\to \R$ satisfies (A2).   Fix $q\in (1, \infty)$ and take $p\in (1, \infty)$ close to $q$. For any such $p$, let $u_p\in C(\Om_T)$ be a solution to \eqref{vp-parabolic} satisfying \eqref{parabolic bdry-cond} with $g=g_p\in C(\partial_p \Om_T)$. 
	If for all $p$ close to $q$,  $u_p$ are equi-H\"older continuous in $\overline{\Om_T}$ with exponent $\theta\in (0, 1]$, then the following results hold:
	\begin{enumerate}
		\item If $p<2$ and $1<q\leq 2$, then for any $-1/2<\beta<p/2-1$, there exists $C>0$ independent of $p$ near $q$ such that
		\begin{equation}\label{rate:p-lap2par-new}
			\sup_{\overline{\Omega}\times[0,T)}|u_p-u_q|\leq \sup_{\partial_p\Omega_T}|g_p-g_q| +C|p-q|^{\theta}. 
		\end{equation}
		\item If $p>2$ and $q\geq 2$, then for any $\beta>q/2-1$, there exists $C>0$ independent of $q$ near $p$ such that 
		\begin{equation}\label{rate:p-lap2par}
			\sup_{\overline{\Omega}\times[0,T)}|u_p-u_q|\leq \sup_{\partial_p\Omega_T}|g_p-g_q| +C|p-q|^{\frac{\theta}{1+(1-\theta)\beta}}. 
		\end{equation}
	\end{enumerate}
\end{thm}
As explained in Remark~\ref{rmk holder}, for the case (2), one may relax the regularity of $u_\vep$, assuming that they are equi-H\"older continuous only in $x$. 

As an immediate consequence of Theorem \ref{thm:p-lap2par}, for equi-Lipschitz solutions $u_p$ of \eqref{vp-parabolic}, the rate for the uniform convergence of $u_p\to u_q$ is $O(|p-q|)$ if $u_p$ and $u_q$ have the same boundary value $g_p=g_q$ on $\partial_p \Omega_T$. Below let us include an example about the Barenblatt solutions that reflects this convergence rate for the parabolic variational $p$-Laplace equation. 
\begin{example}
	Let $p>\frac{2n}{n+1}$, $p\not=2$ and consider the parabolic $p$-Laplace equation 
	\begin{equation}\label{vp-parabolic2}
		\partial_t u-\Delta_p u=0 
	\end{equation}
	in $\R^n\times (0, \infty)$. We examine the Barenblatt solution
	\begin{equation}
		\label{eq:Barenblatt}
		\mathcal{B}_p(x,t)=t^{-\frac{n}{\lambda_p}}\left(A+\gamma_p\left(\abs{x}t^{-\frac{1}{\lambda_p}}\right)^{\frac{p}{p-1}}\right)_+^{\frac{p-1}{p-2}}
	\end{equation}
	where $A>0$ is arbitrary, $\lambda_p=n(p-2)+p$, and
	\begin{equation*}
		\gamma_p=\frac{2-p}{p}\lambda_p^{\frac{1}{1-p}}>0. 
	\end{equation*}
	The Barenblatt solutions were initially constructed as fundamental solutions for the porous medium equation in \cite{Barenblatt1952}; see the book \cite{Barenblatt1996} for details. These were modified to construct fundamental solutions for \eqref{vp-parabolic2} in \cite{Kamin1988}, where they also prove their uniqueness. Barenblatt-type solutions play a key role in many proofs using comparison principle. These papers examine weak solutions of \eqref{vp-parabolic2} but these are equivalent to viscosity solutions for all $1<p<\infty$ as shown in \cite{Juutinen2001}, see also \cite{Siltakoski2021}. 
	

	In the so-called supercritical fast-diffusion regime, where  ${\frac{2n}{n+1}<p<2}$, we have $\gamma_p>0$, while in the degenerate (slow diffusion) case, where $p>2$, one has $\gamma_p<0$. In either case, it is known that $\mathcal{B}_p$ is a viscosity solution of \eqref{vp-parabolic2} in $\R^n\times (0, \infty)$.   
	
	By direct computations, at any $(x, t)\in \R^n\times (0, \infty)$ such that 
	\[
	h(x, t, p)=A+\gamma_p\left(\abs{x}t^{-\frac{1}{\lambda_p}}\right)^{\frac{p}{p-1}}>0, 
	\]
	we have 
	\begin{align*}
		\frac{\partial}{\partial p} \mathcal{B}_p(x,t)=\mathcal{B}_p(x,t) &\left(\frac{n(n+1)}{\lambda_p^2} \log t -\frac{1}{(p-2)^2}\log h(x,t,p)\right.\\&\quad\quad +\left.\left(\frac{p-1}{p-2}\right)h(x, t, p)^{-1}\frac{\partial}{\partial p}h(x, t, p)\right).	
	\end{align*}
	Furthermore
	\begin{align*}
		\frac{\partial}{\partial p}h(x, t, p)&=\gamma_p\left(\abs{x}t^{-\frac{1}{\lambda_p}}\right)^{\frac{p}{p-1}} \frac{\partial}{\partial p}\left(\log\gamma_p+\frac{p}{(p-1)}\left[\log|x|-\frac{1}{\lambda_p}\log t\right]\right)\\
		&=\gamma_p\left(\abs{x}t^{-\frac{1}{\lambda_p}}\right)^{\frac{p}{p-1}} \left(\gamma_p^{-1}\frac{\partial}{\partial p}\gamma_p-\frac{1}{(p-1)}\log|x|+\frac{(n+1)p^2-2n}{(p-1)^2\lambda_p^2}\log t\right),
	\end{align*}
	where
	\begin{equation*}
		\gamma_p^{-1}\frac{\partial}{\partial p}\gamma_p=-\frac{2}{p(2-p)}+\frac{1}{(1-p)^2}\log \lambda_p+\frac{n+1}{(1-p)\lambda_p}.
	\end{equation*}
	Fix $q>1$ and $\delta>0$ such that $[q-\delta,q+\delta]$ is contained entirely in one of the regimes described above. If $\frac{2n}{n+1}<q<2$, one observes from the calculations above that for any $p\in [q-\delta, q+\delta]$ and for any compactly contained open subset $\O \Subset \R^n\times (0, \infty)$, there exists $C>0$ such that 
	\[
	\left|\frac{\partial}{\partial p} \mathcal{B}_p\right|\leq C \quad \text{a.e. in $\O$.} 
	\]
	The upper bound blows up on either end of the regime. In the degenerate case $q>2$, we need to be careful to avoid the free boundary and choose $\O\Subset \R^n\times (0, \infty)$ so that $h(x,t,p)>\delta$ for all $(x,t)\in\O$ and $p\in[q-\delta,q+\delta]$ to get a similar bound. In particular in either regime, by choosing $\Omega_T:=B_\rho(0)\times (\tau, T)$ for suitable $\rho>0$ and $0<\tau<T$, we obtain $C>0$ such that
	\[
	\sup_{\Omega_T}\abs{\mathcal{B}_p-\mathcal{B}_q}
	\leq C\abs{p-q}
	\]
	holds for all $p>1$ sufficiently close to $q$. 
	
	For each $p$, we can view $\mathcal{B}_p$ as the unique solution of \eqref{vp-parabolic2} satisfying the parabolic boundary data $g_p=\mathcal{B}_p$ on $(B_\rho(0)\times \{\tau\})\cup (\partial B_\rho(0)\times [\tau, T))$. Note that $\mathcal{B}_p$ are equi-Lipschitz in $\Omega_T=B_\rho(0)\times (\tau, T)$ for all $p$ near $q$. We thus see that Theorem \ref{thm:p-lap2par} (with $\theta=1$) gives the optimal convergence rate in this case.
\end{example}

One can extend the results in Theorem~\ref{thm:p-lap2par} to the case of the elliptic $p$-Laplacian, following Theorem~\ref{thm:elliptic}. Below we provide a simple example. 

\begin{example}
	Let $p>1$ and $\lambda_p>0$ be a constant given by 
	\[
	\lambda_p=(n+1)\left(\frac{p+1}{p-1}\right)^{p-1}. 
	\]
	Consider the following quasilinear elliptic equation 
	\begin{equation}\label{eq:elliptic-ex1}
		\lambda_p u^{\frac{p-1}{p+1}} -\Delta_p u =0  \quad \text{in $B_1(0)$,}
	\end{equation}
	with the Dirichlet boundary condition 
	\begin{equation}\label{eq:elliptic-ex2}
		u =1 \quad \text{on $\partial B_1(0)$.}
	\end{equation}
	One can verify that, for each $p>1$, 
	\begin{equation}\label{eq:elliptic-ex3}
		u_p(x)=|x|^{\frac{p+1}{p-1}}, \quad x\in \overline{B_1(0)}
	\end{equation}
	is a solution of the Dirichlet problem \eqref{eq:elliptic-ex1}\eqref{eq:elliptic-ex2}. In fact, for 
	\[
	u(x,t)=C\abs{x}^{r} \quad r>0, C\in \R,
	\]
	by direct computations, we have, at any $x\in B_1(0)\setminus \{0\}$,
	\[
	\nabla u(x)=Cr |x|^{r-2} x, \quad 
	\nabla^2 u(x)= Cr |x|^{r-2} \left(I+(r-2)\frac{x\otimes x}{|x|^2}\right)
	\]
	which yields
	\begin{equation}\label{eq:plapformula}
		\begin{aligned}
			\Delta_p u(x)& = |\nabla u(x)|^{p-2} \tr\left[\left(I +(p-2) \frac{\nabla u\otimes \nabla u}{|\nabla u|^2}\right) \nabla^2 u(x)\right] \\
			&=C\abs{C}^{p-2} r^{p-1} |x|^{(r-1)(p-1)-1} \tr\left[ \left(I +(p-2) \frac{x\otimes x}{|x|^2}  \right) \left(I+(r-2) \frac{x\otimes x}{|x|^2} \right)\right]\\
			& =C\abs{C}^{p-2} r^{p-1} \left(n-1+(p-1)(r-1)\right) |x|^{(r-1)(p-1)-1}. 
		\end{aligned}
	\end{equation}
	Here, $u_p$ corresponds to the choice $C=1$ and $(p-1)(r-1)-1=1$, and thus
	\begin{equation*}
		\Delta_p u_p(x)=\left(\frac{p+1}{p-1}\right)^{p-1}(n+1)\abs{x}={\lambda_p}\abs{x},
	\end{equation*}
	and it follows immediately that 
	\[
	-\Delta_p u_p(x)=-\lambda_p|x|= -\lambda_p u_p(x)^{\frac{p-1}{p+1}}
	\]
	for all $x\in B_1(0)\setminus \{0\}$. On the other hand, at $x=0$, $u_p$ satisfies the equation in the sense of $\mathcal{F}$-solutions. It is clear that $u_p$ also satisfies the boundary condition \eqref{eq:elliptic-ex2} for all $p>1$. We thus have verified that $u_p$ given by \eqref{eq:elliptic-ex3} is a solution of the Dirichlet problem. Since the elliptic operator in \eqref{eq:elliptic-ex1} also fulfills the assumptions (A1)(A2')(A3), we can actually use Theorem~\ref{thm:elliptic} to show that $u_p$ is the only Lipschitz solution of the Dirichlet problem for each $p>1$. 
	
	Regarding the convergence rate for $u_p\to u_q$ as $p\to q$ for any fixed $q>1$, it is not difficult to see from \eqref{eq:elliptic-ex3} that there exists $C>0$ independent of $p>1$ near $q$ such that
	\[
	\max_{\overline{\Omega}}|u_p-u_q|\leq C|p-q|.
	\]
	This estimate is consistent with our general quantitative stability result in Theorem~\ref{thm:elliptic} for equi-Lipschitz solutions. Indeed, we have $|\lambda_p-\lambda_q|=O(|p-q|)$ for $p$ near $q>1$, which verifies (C2') with $\gamma=1$. The assumption (C1) can be verified using our previous estimates for the parabolic case. 
\end{example}

Based on the computations in \eqref{eq:plapformula}, 
we can also show that 
\[
v_p(x)=-\frac{p-1}{p}\left(\frac{c}{n}\right)^{\frac{1}{p-1}}|x|^{\frac{p}{p-1}}, \quad x\in \overline{B_1(0)}, 
\]
solves 
\[
-\Delta_p v_p =c  \quad \text{in $B_1(0)$}
\]
for any $c>0$, with constant boundary values on $\partial B_1(0)$. Our result in Theorem~\ref{thm:elliptic} does not apply to this example, since the lower order term $H_\vep$ in this equation is constant and therefore does not satisfy (A3). However, the example suggests that the uniform convergence $v_p\to v_q$ in $\Omega$ as $p\to q$ for any fixed $q>1$ still occurs at the rate $O(|p-q|)$. 

Another important example that cannot be treated by Theorem~\ref{thm:elliptic} is related to fundamental solutions of $p$-Laplace equation. Let $\Omega$ be a punctured ball $B_1(0)\setminus \{0\}$ in $\R^n$. Using the formula we proved above \eqref{eq:plapformula}, for each $p>n$,  
\[
w_p(x)=|x|^{\frac{p-n}{p-1}} 
\]
is $p$-harmonic in $\Omega$ satisfying the boundary condition $w_p(0)=0$ and $w_p=1$ on $\partial B_1(0)$. It actually solves the normalized $p$-Laplace equation as well. When $p>n>1$, this function is known as the fundamental solution to the $p$-Laplace equation \cite{Linotes}. Such functions are also called Hölder cones and were used in \cite{Bun} as comparison functions in the proof of their quantitative stability result.

For any fixed $q>n$, the explicit form implies that as $p\to q$, $w_p\to w_q$ uniformly with the rate $O(|p-q|)$. However, Theorem~\ref{thm:elliptic} does not apply to this equation, since (A3) fails to hold again. It would be an interesting future direction to develop a PDE-based method that yields convergence rates for a more general class of quasilinear elliptic problems, including this example and the one discussed above.

\subsection{General $p$-Laplace type equations}
Next we look into
\begin{equation}
	\label{eq:pq-parabolic}
	\partial_t u-\abs{\nabla u}^{p'-p}\div\left(\abs{\nabla u}^{p-2}\nabla u\right)=0
\end{equation}
for $p'>1$ and $p>1$, which represents a broad class of non-divergence type equations that generalize both the variational $p$-Laplacian and the normalized $p$-Laplacian. This equation was given as an example when Ohnuma and Sato \cite{OS}  introduced a notion of viscosity solutions to singular parabolic equations. 

We are interested in the stability of the solution with respect to the exponent pair $(p, p')$.

Let us fix $q, q'\in (1, \infty)$ and take
\begin{equation*}
	A_{\vep}(\xi)=\abs{\xi}^{p'-2}\left(I+(p-2)\frac{\xi\otimes \xi}{\abs{\xi}^2}\right)
\end{equation*}
\begin{equation}\label{general p-lap}
	A_0(\xi)=\abs{\xi}^{q'-2}\left(I+(q-2)\frac{\xi\otimes \xi}{\abs{\xi}^2}\right), 
\end{equation}
for $(p, p')$ close to $(q, q')$, where we set $\vep=\max\{|p-q|, |p'-q'|\}$. By direct computations, we get
\begin{align*}
	&A_\vep(\xi)^{1/2}=|\xi|^{\frac{p'}{2}-1}\left(I+(\sqrt{p-1}-1) \frac{\xi\otimes \xi}{|\xi|^{2}}\right)
	&A_0(\xi)^{1/2}=|\xi|^{\frac{q'}{2}-1}\left(I+(\sqrt{q-1}-1) \frac{\xi\otimes \xi}{|\xi|^{2}}\right),\\
\end{align*}
for $\xi\in \R^n_0$. 
It follows that
\[
\abs{A_\vep(\xi)^{1/2} -A_0(\xi)^{1/2}}\leq
\abs{\xi}^{\frac{q'}{2}-1}\abs{\sqrt{p-1}-\sqrt{q-1}}+ (p+1)\left||\xi|^{\frac{p'}{2}-1}-|\xi|^{\frac{q'}{2}-1}\right|,
\]
and thus 
there exists $c_A>0$ such that 
\begin{equation}\label{A-est-nondivp}
	\abs{A_\vep(\xi)^{1/2} -A_0(\xi)^{1/2}}
	\leq
	c_A\abs{p-q}\left(1+\abs{\xi}^{\frac{q'}{2}-1}\right)+c_A\abs{p'-q'}\left(1+|\xi|^{\beta}\right)
\end{equation}
for all $\xi\in\R^n_0$ and $\beta$ satisfying
\begin{equation}\label{beta-choice}
	\begin{aligned}
		-\frac{1}{2}<\beta< \frac{q'}{2}-1 &\quad \text{if $q'\leq 2$, $p'<2$,}\\
		\beta> \frac{q'}{2}-1 &\quad \text{if $q'\geq 2$, $p'>2$.}
	\end{aligned}
\end{equation}
This yields condition (C1) with $\alpha=1$ and $\beta$ given by \eqref{beta-choice}. 

Note that when $p'=q'$, \eqref{A-est-nondivp} reduces to
\[
\abs{A_\vep(\xi)^{1/2} -A_0(\xi)^{1/2}}\leq  c_A\abs{p-q}\left(1+\abs{\xi}^{\frac{q'}{2}-1}\right),
\]
which corresponds to (C1) with 
\[
\alpha=1, \quad \beta=\frac{q'}{2}-1.
\]

For this class of parabolic equations, existence and uniqueness of $\mathcal{F}$-solutions to the Cauchy problem for \eqref{eq:pq-parabolic} are available, since we can choose $|x|^k$ with $k>p'/(p'-1)$ as the base of our test class for the definition of $\mathcal{F}$-solutions. We refer to \cite{OS} for uniqueness and to \cite{Gbook} for existence; alternatively, both results fall within the general framework developed in \cite{De11}, where existence and uniqueness are established for a broader class of equations. Consequently, for such choices of $k$, Theorem~\ref{thm:parabolic} applies and leads to the following result.

\begin{thm}\label{thm:pq-lap1par}
	Let $\Omega\subset \R^n$ be a bounded domain and $T>0$. Fix $(q, q')\in (1, \infty)^2$ and take $(p, p')$ close to $(q, q')$. For any such $(p, p')$, let $u_{p,p'}\in C(\Om_T)$ be a solution to \eqref{eq:pq-parabolic} satisfying \eqref{parabolic bdry-cond} with $g=g_{p,p'}\in C(\partial_p \Om_T)$. 
	Assume in addition that for all such $(p, p')$, $u_{p,p'}$ are equi-H\"older continuous in $\overline{\Om_T}$ with exponent $\theta\in (0, 1]$. Let $\beta>-1/2$ be an arbitrary value satisfying \eqref{beta-choice}. Then the following results hold:
	\begin{enumerate}
		\item For $p'> 2$ and $q'\geq 2$, there exists $C>0$ independent of $(p, p')$ such that
		\begin{equation}\label{rate:pq-lap1par}
			\sup_{\ol{\Omega}\times[0,T)}|u_{p,p'}-u_{q,q'}|\leq \sup_{\partial_p\Omega_T}|g_{p,p'}-g_{q,q'}| +C(|p-q|+|p'-q'|)^{\theta}. 
		\end{equation}
		\item For $1<p'< 2$ and $1<q'\leq 2$, there exists $C>0$ independent of $(p, p')$ such that
		\begin{equation}\label{rate:pq-lap1par2}
			\sup_{\ol{\Omega}\times[0,T)}|u_{p,p'}-u_{q,q'}|\leq \sup_{\partial_p\Omega_T}|g_{p,p'}-g_{q,q'}| +C(|p-q|+|p'-q'|)^{\frac{\theta}{1+(1-\theta)\beta}}. 
		\end{equation}
		\item In the special case that $p'= q'\in (1, \infty)$, there exists $C>0$ independent of $q$ such that 
		\begin{equation}\label{rate:pq-lap2par}
			\sup_{\ol{\Omega}\times[0,T)}|u_{p,p'}-u_{q,p'}|\leq \sup_{\partial_p\Omega_T}|g_{p,p'}-g_{q,p'}| +C|p-q|^{\nu}, 
		\end{equation}
		holds with 
		\begin{equation}\label{nu-pq}
			\nu=\frac{2\theta}{2\theta+(1-\theta){q'}}. 
		\end{equation}
	\end{enumerate}
\end{thm}
In view of Remark \ref{rmk holder} and Theorem \ref{thm:parabolic2}, in the case $p'> 2$, $q'\geq 2$,  we can drop the equi-H\"older continuity of $u_{p, p'}$ with respect to the time variable and only keep it for the space variable.

\subsection{Regularization for generalized $p$-Laplace equations}
Another useful application is to look into the regularization of the generalized $p$-Laplace equations. For a bounded domain $\Omega\subset \R^n$ and $T>0$, let us consider the general equation \eqref{eq:pq-parabolic}, which corresponds to $A_0$ given in \eqref{general p-lap} with $p\geq 1, p'\geq 2$. For $\vep>0$, we consider the regularized equation \eqref{regularized eq}. The associated operator $A_\vep$ is given by 
\begin{equation*}
	A_\vep(\xi)=\left(|\xi|^2+\eps^2\right)^{\frac{p'-2}{2}}\left( I+(p-2) \frac{\xi\otimes \xi}{|\xi|^2+\eps^2}\right). 
\end{equation*}
It is clear that $A_\vep$ is bounded for all $\vep\geq 0$, thanks to the condition $p'\geq 2$. 

Suppose that for $\vep\geq 0$ and given boundary data $g_\vep\in C(\partial_p\Omega_T)$, the solution $u_\vep$ of \eqref{regularized eq} is equi-H\"older continuous in space in the sense of \eqref{equi-holder-x}. We can apply Theorem \ref{thm:parabolic2} to estimate the difference between $u_\vep$ and $u_0$. Note that the condition $p'\geq 2$ enables us to take $k>2$ arbitrarily in (A1). 

For $\vep\geq 0$, let
\[
\tilde{A}_\vep(\xi)=I+(p-2) \frac{\xi\otimes \xi}{|\xi|^2+\eps^2}, \quad \xi\in \R^n_0.
\]
A direct computation yields
\[
\tilde{A}_\vep(\xi)^{\frac{1}{2}}=I+\frac{1}{|\xi|^2} \left(\sqrt{\frac{|\xi|^2(p-1)+\vep^2}{|\xi|^2+\vep^2}}-1\right)\xi\otimes \xi.  
\]
Then, it follows that 
\begin{equation}\label{regularization eq0}
	\begin{aligned}
		&\left\|\tilde{A}_\vep(\xi)^{\frac{1}{2}}-\tilde{A}_0(\xi)^{\frac{1}{2}}\right\|\leq \left|\sqrt{\frac{|\xi|^2(p-1)+\vep^2}{|\xi|^2+\vep^2}}-\sqrt{p-1}\right|\\
		&\leq \vep^2|p-2| \left(|\xi|^2+\vep^2\right)^{-\frac{1}{2}}\left(\left(|\xi|^2(p-1)+\vep^2\right)^{\frac{1}{2}}+\left((|\xi|^2+\vep^2)(p-1)\right)^{\frac{1}{2}}\right)^{-1} \\
		&\leq \vep|p-2|(|\xi|^2+\vep^2)^{-\frac{1}{2}}.
	\end{aligned}    
\end{equation}
Thus estimating the original matrices, we have
\begin{equation}\label{regularization eq1}
	\left\|A_\vep(\xi)^{\frac{1}{2}}-A_0(\xi)^{\frac{1}{2}}\right\|\leq \begin{dcases}
		C_A\vep(|\xi|^2+\vep^2)^{-\frac{1}{2}} & \text{if $p'=2$,}\\
		C_A\vep^{p'-2}+C_A\vep|\xi|^{\frac{p'-2}{2}}(|\xi|^2+\vep^2)^{-\frac{1}{2}} & \text{if $2<p'\leq 4$,}\\
		C_A\vep^2\left(1+|\xi|^{p'-4}\right)+C_A\vep|\xi|^{\frac{p'-2}{2}}(|\xi|^2+\vep^2)^{-\frac{1}{2}} & \text{if $p'>4$}
	\end{dcases}  
\end{equation}
for some $C_A>0$ independent of $\vep>0$ and $\xi\in \R^n_0$. Note that by Young's inequality, we get
\begin{equation}\label{regularization eq2}
	|\xi|^2+\vep^2\geq \left(\frac{1}{m}\right)^m\left(\frac{1}{1-m}\right)^{1-m}|\xi|^{2m} \vep^{2(1-m)} \geq |\xi|^{2m} \vep^{2(1-m)}
\end{equation}
for any $0<m<1$. Based on \eqref{regularization eq1} and \eqref{regularization eq2}, let us discuss the convergence rate of $u_\vep\to u_0$ for different $p'$.

\noindent Case 1: $p'=2$. By \eqref{regularization eq1} and \eqref{regularization eq2}, we are led to 
\[
\left\|A_\vep(\xi)^{\frac{1}{2}}-A_0(\xi)^{\frac{1}{2}}\right\|\leq C_A\vep^{m}(1+|\xi|^{-m}).
\]
To ensure that the assumption (C1) holds in this case, we choose $m<1/2$, so that 
\[
-m>\frac{2-k}{2(k-1)}
\]
for $k>2$ sufficiently large. With this choice, (C1) holds with $\alpha=m$, $\beta=-m$. Applying Theorem \ref{thm:parabolic2}, we obtain the estimate below 
\begin{equation}\label{regularization est2}
	\sup_{\ol{\Omega}\times[0,T)} |u_\vep-u_0|\leq \sup_{\partial_p \Om_T} |g_\vep-g_0|+C\vep^{m\theta}. 
\end{equation}
for any $m\in (0, 1/2)$, under the equi-H\"older regularity \eqref{equi-holder-x}. 

\noindent  Case 2: $2<p'\leq 4$. In this case, \eqref{regularization eq1} and \eqref{regularization eq2} yield 
\[
\begin{aligned}
	\left\|A_\vep(\xi)^{\frac{1}{2}}-A_0(\xi)^{\frac{1}{2}}\right\|
	&\leq C_A \vep^{p'-2}+C_A\vep^{m}(1+|\xi|^{-m-1+\frac{p'}{2}}) \\
	&\leq 2C_A \vep^{\alpha} (1+|\xi|^{-m-1+\frac{p'}{2}})
\end{aligned}
\]
with $\alpha=\min\{p'-2, m\}$, where we need to take $m<(p'-1)/2$ to ensure (C1). 
Hence, Theorem~\ref{thm:parabolic2} implies the convergence rate
\begin{equation}\label{regularization est1}
	\sup_{\ol{\Om}\times[0,T)} |u_\vep-u_0|\leq \sup_{\partial_p \Om_T} |g_\vep-g_0|+C\vep^{\nu}
\end{equation}
for $\nu=\theta\min\{p'-2, m\}$,  
if $u_\vep$ are spatially equi-H\"older continuous solutions with exponent $\theta\in (0, 1]$.

\noindent  Case 3: $p'> 4$. We then use \eqref{regularization eq1} and \eqref{regularization eq2} to obtain 
\[
\begin{aligned}
	\left\|A_\vep(\xi)^{\frac{1}{2}}-A_0(\xi)^{\frac{1}{2}}\right\|
	&\leq C_A\vep^2(1+|\xi|^{p'-4})+C_A\vep^{m}(1+|\xi|^{-m-1+\frac{p'}{2}}) \\
	&\leq 2C_A \vep^{m} (1+|\xi|^{\beta})
\end{aligned}
\]
with 
\[
\beta=\max\left\{p'-4, -m-1+\frac{p'}{2}\right\}>0.
\]
For this $\beta$, it then follows from Theorem \ref{thm:parabolic2}  that $u_\vep\to u_0$ at the rate in \eqref{regularization est1} with 
\[
\nu= \frac{m\theta}{1+(1-\theta)\beta}.
\]
It is not difficult to see that such $\nu$ is nondecreasing with respect to $m$, and thus the estimate is optimized as $m\in (0, 1)$ approaches $1$ and $\beta= p'-4$.

Our results can be summarized as follows. 
\begin{thm}\label{thm:regularization}
	Let $\Omega\subset \R^n$ be a bounded domain and $T>0$. Fix $p\geq 1$, $p'\geq 2$. 
	For $\vep\geq 0$, let $u_{\vep}\in C(\Om_T)$ be a solution of \eqref{regularized eq} with boundary data $u_\vep=g_{\vep}\in C(\partial_p \Om_T)$.  Assume in addition that $u_{\vep}$ are equi-H\"older continuous with respect to the space variable in the sense of \eqref{equi-holder-x} for some $L>0$ and $\theta\in (0, 1]$. Then, there exist $C>0$ and $\nu\in (0, 1)$ such that \eqref{regularization est1} holds.  More precise conditions regarding the exponent $\nu$ are as follows. 
	\begin{enumerate}
		\item If $p'=2$, then $0<\nu <\theta/2$;
		\item If $2<p'< 3$, then  $\nu=(p'-2)\theta$;
		\item If $3\leq p'\leq 4$, then  $0<\nu<\theta$;
		\item If $p'>4$, then 
		\[
		0<\nu< \frac{\theta}{1+(1-\theta)(p'-4)}. 
		\]
	\end{enumerate}
\end{thm}

Our result above, in the case of $p'=2$, is consistent with \cite[Theorem 1]{Mi} for Lipschitz solutions to the level set mean curvature flow equation, which corresponds to our setting with $p'=2$, $p=1$, $\theta=1$ and $g_\vep=g_0$. Despite the different settings of boundary conditions, Theorem \ref{thm:regularization} verifies the same convergence rate $O(\vep^\nu)$ with any $\nu\in (0, 1/2)$, as obtained in \cite{Mi}. Our result additionally provides the exponent bound $\theta/2$ when the solutions $u_\vep$ are merely $\theta$-H\"older continuous in space. 

\subsection{Regularization for biased infinity-Laplace equation}
Let us lastly examine the following nonlinear parabolic equation
\begin{equation}
	\label{eq:biased}
	\partial_tu-\Delta_\infty^N u-a \abs{\nabla u}=f(x,t) \quad \text{in $\Omega_T$,}
\end{equation}
where $a\in \R$ is given and $\Delta_\infty^N$ denotes the $\infty$-Laplacian, defined by 
\[
\Delta_\infty^N u:=\frac{1}{\abs{\nabla u}^2}\tr\left[\left(\nabla u \otimes \nabla u\right) \nabla ^2 u\right].
\]
Such equations involving $\infty$-Laplacian are closely related to the study of stochastic game theory \cite{PSSW}. Viscosity solutions of \eqref{eq:biased} can be approximated by value functions of a family of so-called $a$-biased tug-of-war games, as shown in \cite{Peres2010}. Comparison principle for this equation when $f$ does not change sign is proven in \cite{Liu2019}, where the authors also establish existence and stability results using a regularized approximation that fits into our setting. The regularized equation considered in \cite{Liu2019} reads 
\begin{equation}
	\label{eq:biasedreg}
	\partial_tu-\eps_1\Delta u-\frac{\tr(\left(\nabla u \otimes \nabla u\right)\nabla^2 u)}{\sqrt{\abs{\nabla u}^2+\eps_1^2}}-a\sqrt{\abs{\nabla u}^2+\eps_2^2}=f(x,t),
\end{equation}
where $\vep_1, \vep_2>0$ are small parameters. 

Under our general framework, for any $\vep=(\vep_1, \vep_2)$ we set
\begin{equation*}
	A_0(\xi)=\frac{\xi\otimes \xi}{|\xi|^2}, \quad A_\vep(\xi)=\frac{\xi\otimes \xi}{\abs{\xi}^2+\eps_1^2}+\eps_1 I,
\end{equation*}
for $\xi\in \R^n_0$ and
\begin{equation*}
	H_0(x,t,\xi)=-a\abs{\xi}-f(x,t), \quad H_\vep(x,t,\xi)=-a\sqrt{\abs{\xi}^2+\vep_2^2}-f(x,t),
\end{equation*}
for $(x, t)\times \Omega_T$ and $\xi\in \R^n$. The singularity of $A_0(\xi)$ at $\xi=0$ is mild, and (A1) holds for any $k>0$. The assumption (A2) also clearly holds. Moreover, we have $A_0(\xi)^{1/2}=A_0(\xi)$ for all $\xi\in \R^n_0$. 
For $\vep_1>0$, by computations we obtain 
\begin{equation*}
	A_\vep(\xi)=\left(\eps_1 +\frac{\abs{\xi}^2}{\abs{\xi}^2+\eps_1^2}\right)\frac{\xi\otimes \xi}{|\xi|^2}+\eps_1 \left(I-\frac{\xi\otimes \xi}{|\xi|^2}\right),
\end{equation*}
\begin{equation*}
	A_\eps(\xi)^{1/2}=\sqrt{\eps_1 +\frac{\abs{\xi}^2}{|\xi|^2+\eps_1^2}}\frac{\xi\otimes \xi}{|\xi|^2}+\sqrt{\eps_1}\left(I-\frac{\xi\otimes \xi}{|\xi|^2}\right).
\end{equation*}
We thus can estimate
\[
\begin{aligned}
	\norm{A_\eps(\xi)^{1/2}-A_0(\xi)^{1/2}}^2&= \norm{\left(\sqrt{\eps_1 +\frac{\abs{\xi}^2}{|\xi|^2+\eps_1^2}}-1\right)\frac{\xi\otimes \xi}{|\xi|^2} +\sqrt{\eps_1}\left(I-\frac{\xi\otimes \xi}{|\xi|^2}\right)}^2\\
	&= \left(\sqrt{\eps_1 +\frac{\abs{\xi}^2}{|\xi|^2+\eps_1^2}}-1\right)^2+\eps_1\leq \left(\eps_1-\frac{\eps_1^2}{|\xi|^2+\eps_1^2}\right)^2+\eps_1.
\end{aligned}
\]
By Young's inequality,
\begin{equation*}
	\frac{\eps_1^2}{|\xi|^2+\eps_1^2} \leq \frac{\eps_1^2}{\abs{\xi}^{2m}\eps_1^{2-2m}}\leq \eps_1^{2m}|\xi|^{-2m}
\end{equation*}
which combined with above implies
\[
\norm{A_\eps(\xi)^{1/2}-A_0(\xi)^{1/2}}^2\leq 2\eps_1+ \eps_1^{4m}|\xi|^{-4m}
\]
for any $m\in (0, 1)$ and $\eps_1>0$ small. By restricting the choice of $m$ to $(0, 1/4)$,  we obtain
\[
\norm{A_\eps(\xi)^{1/2}-A_0(\xi)^{1/2}}\leq C\eps_1^{2m}(1+|\xi|^{-2m})
\]
for some $C>0$ independent of $\eps_1>0$. Hence, the assumption (C1) holds with $\alpha=2m$ and $\beta=-2m$ for any $m\in (0, 1/4)$. 

As for (C2), we have
\begin{equation*}
	\abs{H_\vep(x, t, \xi)-H_0(x, t, \xi)}=\abs{a\abs{\xi} -a\sqrt{\abs{\xi}^2+\vep_2^2}}\leq \abs{a}\vep_2,
\end{equation*}
so it holds with $\gamma=1$. We can now use Theorem \ref{thm:parabolic} to obtain the following result. 
\begin{thm}
	Let $\Omega\subset \R^n$ be a bounded domain and $T>0$. For $\eps=(\vep_1, \vep_2)$ with $\vep_1, \vep_2\geq0$, let $u_{\eps}\in C(\Om_T)$ be a solution to \eqref{eq:biasedreg} with boundary value $u_\vep=g_{\eps}\in C(\partial_p \Om_T)$. Let $u_0\in C(\Om_T)$ be a solution to \eqref{eq:biased} satisfying $u_0=g_0\in C(\partial_p \Om_T)$.  Suppose that $u_{\eps}$ are equi-H\"older continuous in $\overline{\Om_T}$ with exponent $\theta\in (0, 1]$ for all $\vep_1, \vep_2$ small. Then, there exists $C>0$ such that
	\begin{equation*}
		\sup_{\ol{\Om}\times[0,T)}|u_{\eps}-u_0|\leq \sup_{\partial_p\Omega_T}|g_{\eps}-g_0| +C\eps_1^{\alpha\theta}+C\vep_2
	\end{equation*}
	for any $\alpha\in (0, 1/2)$.
\end{thm}

	\bibliographystyle{abbrv}

\end{document}